\documentclass[12pt]{article}
\UseRawInputEncoding 
\usepackage{amssymb,amsfonts,amsthm}
\usepackage{amsmath,amsthm,amssymb}
\usepackage{amscd}
\usepackage{amsfonts}
\usepackage{color}
\usepackage{mathrsfs}

\newcommand{\vv}{{\mathbb V}}

\newcommand{\con}{\operatorname{con}}

\newcommand{\mul}{\operatorname{mul}}
\newcommand{\simp}{\operatorname{sim}}

\newcommand{\var}{\operatorname{var}}

\newtheorem{theorem}{Theorem}[section]
\newtheorem{ex}[theorem]{Example}
\newtheorem{cor}[theorem]{Corollary}
\newtheorem{sufcon}{Sufficient Condition}
\newtheorem{fact}[theorem]{Fact}
\newtheorem{lemma}[theorem]{Lemma}
\newtheorem{definition}[theorem]{Definition}
\newtheorem{remark}[theorem]{Remark}
\newtheorem{prop}[theorem]{Proposition}

\newtheorem{obs}[theorem]{Observation}

\begin{document}

\date{}

\title{Limit varieties of $J$-trivial monoids}
\author{Olga  B. Sapir}

\maketitle

\begin{abstract} We show that limit varieties of monoids recently discovered by Gusev, Zhang and Luo and their subvarieties
are generated by monoids of the form $M_\tau(W)$ for certain congruences $\tau$ on the free monoid.
The construction $M_\tau(W)$ is a generalization of widely used Dilworth-Perkins construction.
Using this construction, we find explicit generators for  Gusev limit varieties and
 give a short reproof to the fact that  Zhang-Luo limit variety is non-finitely based.
\end{abstract}

\section{Introduction}

 In this article  we consider  monoids and regard them as semigroups equipped with an additional $0$-ary operation that fixes the identity element.  Elements of a countably infinite alphabet $\mathfrak A$ are called {\em letters} and elements of the free monoid $\mathfrak A^*$  are called {\em words}. We use $1$ to denote the empty word, which is the identity element of $\mathfrak A^*$.
An algebra (variety of algebras)  is said to be {\em finitely based} (FB) if there is a finite subset of its identities from which all of its identities may be deduced. Otherwise, it is said to be {\em non-finitely based} (NFB).

In 1969, Perkins \cite{P} found the first sufficient condition under which a semigroup is NFB. By using this condition, he constructed the first two examples of finite NFB semigroups.  The first example was the 6-element Brandt monoid and the second example was the 25-element monoid obtained from the set of words $W = \{abtba, atbab, abab, aat\}$ by using the following construction.

Let $\le$ be the relation on  the free monoid ${\mathfrak A}^*$ given by ${\bf v} \le {\bf u}$ if and only if $\bf v$ is a subword of $\bf u$.
 Given a set of words  $W \subseteq {\mathfrak A}^*$,   $M(W)$ denotes the Rees quotient  of  $\mathfrak A^*$  over the ideal $\mathfrak A^* \setminus  W^{\le}$,  where $W^\le$ is the closure of $W$ under taking subwords.
Then
$$M(W) = \langle W^{\le} \cup \{0\}, \star \rangle,$$ where multiplication $\star$ is given by

$\bullet$ ${\mathbf u} \star {\mathbf v} = \mathbf{uv}$ if   $\mathbf {uv} \in W^{\le}$, otherwise, $\mathbf {u} \star \mathbf {v} = 0$.

 $M(W)$ ($M(W) \setminus \{1\}$) is a $J$-trivial monoid (semigroup) because the free monoid  $\mathfrak A^*$ is $J$-trivial, that is,
${\bf u} ={\bf v}$ whenever ${\bf u} \le {\bf v}$ and ${\bf v} \le {\bf u}$.

The sufficient condition of Perkins involves the requirement that certain words are {\em isoterms} for a semigroup under consideration. A word ${\bf u}$ is said to be an {\em isoterm} \cite{P} for a semigroup $S$ if $S$ does not satisfy any nontrivial identity of the form ${\bf u} \approx {\bf v}$.

Since Perkins introduced the notion of an isoterm, it has been a necessary ingredient
in the majority of many arguments invented by other researchers to establish an absence of the finite identity basis for a semigroup.
A locally finite algebra is said to be {\em inherently not finitely based} (INFB) if any locally finite variety containing it is NFB.
The syntactic part of the celebrated description of INFB finite semigroups given by Mark Sapir \cite{MS} states that
a finite semigroup $S$ is INFB  if and only if every Zimin word (${\bf Z}_1=x_1, \dots, {\bf Z}_{k+1} = {\bf Z}_kx_{k+1}{\bf Z}_k, \dots$)  is an isoterm for $S$.
This result implies that the Brandt monoid is INFB while finite $J$-trivial monoids are never INFB.

For the majority of $J$-trivial monoids which are known to be NFB, their property of having no finite identity basis is a consequence of
a sufficient condition of the following form.

 \begin{sufcon} \label{gen1} Let $\Sigma$ be a certain set of identities without any bound on the number of letters involved and $W$ be
 a subset of  ${\mathfrak A}^*$.
If a monoid $M$ satisfies all identities in $\Sigma$ and all the words in $W$ are isoterms for $M$, then $M$ is NFB.
\end{sufcon}

About ten years ago, Lee suggested to investigate the finite basis property of semigroups
\[
L_\ell = \langle e,f \mid e^2=e, f^2=f, \underbrace{efefef \cdots}_{\text{length }\ell} = 0 \rangle, \quad \ell \geq 2
\]
and the monoids $L_\ell^1$ obtained by adjoining an identity element to $L_\ell$.

The 4-element semigroup $L_2 =A_0$ is  long known to be finitely based \cite{1980}.
Zhang and  Luo proved \cite{WTZ1} that the 6-element semigroup $L_3$ is NFB and  Lee generalized this
result into a sufficient condition  \cite{EL} which implies that for all $\ell \geq 3$, the semigroup $L_\ell$ is NFB  \cite{EL1}.

 As for the monoids $L_\ell^1$, the 5-element monoid $L_2^1=A_0^1$ was also proved to be FB by Edmunds \cite{1977}
 and it is shown in \cite{WTZ, OS, MiS} that  for each $\ell \ge 3$ the monoid $L^1_\ell$ is NFB.

In \cite{OS}, we generalized the notion of an isoterm into
a notion of a {\em $\tau$-term}, because  isoterms
give no information about the finite basis  property of Lee monoids (see Section 6 in \cite{OS}).
If  $\tau$ is an equivalence relation on the free semigroup  $\mathfrak A^+$ then a word ${\bf u}$ is said to be  a {\em $\tau$-term} for a semigroup $S$ if ${\bf u} \tau {\bf v}$ whenever $S$ satisfies ${\bf u} \approx {\bf v}$.

 Given a congruence $\tau$ on the free monoid  $\mathfrak A^*$,  we use $\circ{_\tau}$ to denote the binary operation on the quotient monoid $\mathfrak A^*/\tau$.  The elements of  $\mathfrak A^*/\tau$ are called {\em $\tau$-words} or {\em $\tau$-classes}
and written using lowercase letters in the typewriter style\footnote{We call the elements of $\mathfrak A^*/\tau$ by {\em $\tau$-words} when we want to emphasize the 
 relations between them and we refer to  ${\mathtt u} \in \mathfrak A^*/\tau$
as a {\em $\tau$-class} when we are interested in the description of the words contained in  ${\mathtt u}$.}. 
We extend the definition of a $\tau$-term  to $\tau$-classes of $\mathfrak A^+$  as follows:
a $\tau$-class  ${\mathtt u}$ is said to be a {\em $\tau$-term} for $S$ if every word in  ${\bf u} \in {\mathtt u}$ is a $\tau$-term for $S$.

 Let $\tau_1$ denote the congruence on the free monoid ${\mathfrak A}^*$ induced by the relations $a=a^2$ for each $a \in{\mathfrak A}$.
 Then $\mathtt{a^{1+}}= \{a^n \mid n \ge 1\}$ is an idempotent of ${\mathfrak A}^*/ {\tau_1}$ and every element of  ${\mathfrak A}^+/ {\tau_1}$ can be represented as a word in the alphabet $\{{\mathtt a}^{1+} \mid a \in \mathfrak A\}$. For example,
 $(\mathtt{a^{1+} \circ_{\tau_1}b^{1+}\circ_{\tau_1}a^{1+}}) = \{a^m b^n a^k \mid m,n,k \ge 1\}$.
Using $\tau_1$-terms, the sufficient condition of Lee \cite{EL} for semigroups can be reformulated as follows.

\begin{sufcon} \label{EL} (cf. \cite[Theorem 6.2]{MiS})
Suppose that a semigroup $S$ satisfies
$$\Sigma = \{xy^2_1y^2_2 \dots y^2_nx  \approx  xy^2_ny^2_{n-1} \dots y^2_1x \mid n \ge 2\}$$
and  the $\tau_1$-word $(\mathtt{a^{1+}\circ_{\tau_1} b^{1+} \circ_{\tau_1}a^{1+}})$ is a $\tau_1$-term for $S$, then $S$ is NFB.
 \end{sufcon}

 The NFB property of every Lee monoid  $L^1_\ell$ for $\ell \geq 3$ follows from one of the three  sufficient conditions in \cite{MiS, OS1, OS}. Each of these sufficient conditions
 can be obtained from the following general sufficient condition by substituting a specific set of identities for $\Sigma$ and 
by substituting either   $(\mathtt{a^{1+} \circ_{\tau_1}b^{1+}\circ_{\tau_1}a^{1+}})$ or  $(\mathtt{a^{1+} \circ_{\tau_1}b^{1+}\circ_{\tau_1}a^{1+} \circ_{\tau_1}b^{1+}})$ or 
$(\mathtt{a^{1+} \circ_{\tau_1}b^{1+}\circ_{\tau_1}a^{1+} \circ_{\tau_1}b^{1+}\circ_{\tau_1}a^{1+}})$
 for $\mathtt{u}$.

 \begin{sufcon} \label{gen} Let $\Sigma$ be  certain set of identities without any bound on the number of letters involved.
If a monoid $M$ satisfies all the identities in $\Sigma$ and  certain  $\tau_1$-word  $\mathtt{u} \in{\mathfrak A}^*/{\tau_1}$ is a $\tau_1$-term for $M$, then $M$ is NFB.
\end{sufcon}

In 1944, Morse and Hedlund \cite{MH} constructed an infinite set of words $W$ (`unending sequence" \cite{MH}), in a three-letter alphabet with the property that no word in $W$ has a subword of the form ${\bf uu}$.
As mentioned in the introduction of \cite{MH}, Dilworth pointed out that `the existence of
unending sequences" with certain properties
`makes possible the construction of useful examples of semigroups".
In particular, this set of words $W$ yields an infinite semigroup $S$  `generated
by three elements such that the square of every element in $S$ is zero".
So, in essence, the $M(W) \setminus \{1\}$ construction was also used in \cite{MH} to obtain the first example of an infinite finitely generated nil-semigroup.

In 2005, Jackson \cite{MJ} used the  Dilworth-Perkins construction to give the first explicit examples of {\em  limit varieties} of monoids in the sense that each of these varieties is NFB while each proper monoid subvariety of each of these varieties is FB.
Jackson's limit varieties are generated by
$M(\{atbasb\})$ and  $M(\{abtasb, atbsab\})$.

 Let $H_\tau$ denote the natural homomorphism $\mathfrak A^* \rightarrow \mathfrak A^*/\tau$ corresponding to a congruence $\tau$.
In \cite{OS}, we generalised the Dilworth-Perkins construction $M(W)$ into $M_\tau(W)$ construction as follows.

\begin{definition}  \label{D: MtauW}
Let $\tau$ be a congruence on the free monoid $\mathfrak A^*$  and
$W \subseteq \mathfrak A^*$ be a set of words closed under $\tau$. 
Then $M_\tau(W)$ denotes the Rees quotient  of  $\mathfrak A^*/\tau$  over the ideal $H_\tau(\mathfrak A^* \setminus  W^{\le})$.
\end{definition}

For example, according to \cite{OS}, the 5-element monoid
\begin{equation} \label{A01} A_0^1 = L^1_2 = \langle 1, e, f \mid e^2=e, f^2=f, ef = 0 \rangle\end{equation}
 is isomorphic to $M_{\tau_1}(W)$, where 
$W = (\mathtt{a^{1+} \circ_{\tau_1} b^{1+} }) = \{a^m b^n \mid m,n \ge 1\}$.

To simplify the computations involving finite monoids, we redefine the generalized Dilworth-Perkins construction in terms of the (finite)
subsets of $\mathfrak A^* / \tau$ instead of the (infinite) subsets of  $\mathfrak A^*$. The subsets of $\mathfrak A^* / \tau$ are denoted by  capital letters in the typewriter style.
According to Definition~\ref{D: MtauW1},  $A_0^1$ is isomorphic to
 $M_{\tau_1}(\mathtt W)$, where
$\mathtt W = \{\mathtt{a^{1+}\circ_{\tau_1} b^{1+} }\}$ consists of a single $\tau_1$-word.

 Lemma 7.1 in \cite{OS} and its proof give us the following connection between monoids of the form $M_\tau(W)$ and $\tau$-terms for monoid varieties.

\begin{fact} \label{prec}
Let $\tau$ be a congruence on the free monoid $\mathfrak A^*$ such that the empty word  forms a singleton $\tau$-class.
Let $W \subseteq \mathfrak A^*$ be a set of words  which is a union of $\tau$-classes and is closed under taking subwords.
Then a monoid variety $\vv$ contains $M_\tau(W)$ if and only if every word in $W$ is a $\tau$-term for $\vv$.
\end{fact}

Fact \ref{prec} generalizes  Lemma 3.3 in \cite{MJ} which gives a connection between monoids of the form $M(W)$ and isoterms for
monoid varieties.

A letter is called {\em simple} ({\em multiple}) in a word $\bf u$ if it occurs in $\bf u$ once (at least
twice). The set of all (simple, multiple) letters in $\bf u$ is denoted by $\con({\bf u})$
(respectively, $\simp ({\bf u})$ and $\mul({\bf u}))$. 
We partition the congruence $\tau_1$ into three more congruences as follows.
 Given  ${\bf u}, {\bf v} \in \mathfrak A^*$, we define:
\begin{itemize}
\item ${\bf u} \gamma {\bf v}$ if and only if ${\bf u} \tau_1 {\bf v}$ and  $\mul({\bf u}) =  \mul({\bf v})$;
\item ${\bf u} \lambda {\bf v}$ if and only if ${\bf u} \gamma {\bf v}$ and the first two occurrences
of each  multiple  letter are adjacent in $\bf u$ if and only if these occurrences are  adjacent in $\bf v$;
\item ${\bf u} \rho {\bf v}$ if and only if ${\bf u} \gamma {\bf v}$ and the last two occurrences
of each  multiple  letter are adjacent in $\bf u$ if and only if these occurrences are  adjacent in $\bf v$.
\end{itemize}

Notice that the relations $\rho$ and $\lambda$ are dual to each other.
If $\tau \in \{\gamma, \lambda, \rho\}$ then each letter of the alphabet forms a  singleton $\tau$-class
$\mathtt{a} = \{a\}$ and $\mathtt{a^{2+}} = \{a^2, a^3, \dots\}$ is an idempotent of $\mathfrak A^*/\tau$ for each $a \in \mathfrak A^*$.
Since $\tau$ is a congruence,  every element of  ${\mathfrak A}^*/ \tau$ can be represented (in multiple ways) as a word in the alphabet $\{{\mathtt a}, {\mathtt a}^{2+} \mid a \in \mathfrak A\}$. For example,
\[ (\mathtt{a \circ_\lambda t \circ_\lambda b\circ_\lambda a^{2+}\circ_\lambda s \circ_\lambda b^{2+}}) =(\mathtt{a \circ_\lambda t \circ_\lambda b\circ_\lambda a\circ_\lambda s \circ_\lambda b}) = \{atba^{m}sb^{n} \mid m,n \ge 1\}\]
 is a $\lambda$-class of  $\mathfrak A^*$. 
We choose $(\mathtt{a \circ_\lambda t \circ_\lambda b\circ_\lambda a^{2+}\circ_\lambda s \circ_\lambda b^{2+}})$  as the canonical name of this $\lambda$-class.

 In  Proposition~\ref{congruence},  we show 
that  for each $\tau \in \{\gamma, \lambda, \rho\}$,
every element of  ${\mathfrak A}^*/ \lambda$ has a unique representation as a canonical word in the alphabet $\{{\mathtt a}, {\mathtt a}^{2+} \mid a \in \mathfrak A\}$.  Using this representation,  we show 
 that  given a finite set of $\tau$-words $\mathtt W \subset {\mathfrak A}^*/ \tau$, 
 $M_\tau(\mathtt W)$ is finite  $J$-trivial monoid and provide a simple algorithm for computing its  multiplication table. 
In particular, according to Example~\ref{E: J}, the monoid  $M_{\lambda} (\mathtt{a \circ_\lambda t \circ_\lambda b\circ_\lambda a^{2+}\circ_\lambda s \circ_\lambda b^{2+}})$ has 31 elements.

Given a monoid variety $\vv$  we use $\overline{\vv}$ to denote the variety {\em dual to} $\vv$, that is, the variety consisting of monoids anti-isomorphic to the monoids from $\vv$.  
Recently, Gusev \cite{SG} found two new limit varieties of monoids: $\mathbb J$ and $\overline{\mathbb J}$. The variety $\mathbb J$ is given by a certain (infinite) set of identities. While looking through his proof that $\mathbb J$ is NFB, we noticed that it yields
the following sufficient condition:

\begin{sufcon} \label{SG}
Suppose that a monoid  $M$ satisfies
\[\{x y_1 y_2 \dots y_n x t_1 y_1 t_2 y_2 \dots t_n y_n \approx x^2 y_1 y_2 \dots y_n  t_1 y_1 t_2 y_2 \dots t_n y_n \mid n \ge 1\}\]
and  the $\lambda$-word   $(\mathtt{a \circ_\lambda t \circ_\lambda b\circ_\lambda a^{2+}\circ_\lambda s \circ_\lambda b^{2+}})$ is a $\lambda$-term for $M$, then $M$ is NFB.
 \end{sufcon}

 Sufficient Condition \ref{SG} together with Fact \ref{prec} gave us a clue that the variety $\mathbb J$ is generated by $M_{\lambda} (\mathtt{a \circ_\lambda t \circ_\lambda b\circ_\lambda a^{2+}\circ_\lambda s \circ_\lambda b^{2+}})$ (see Theorem \ref{main}(v) below).

Let $A$ denote the monoid obtained by adjoining an identity element to the following semigroup:
\begin{equation}\label{A}
A=\langle e,f,c\mid e^2=e,\,f^2=f,\,ef=ce=0,\, ec=cf=c\rangle=\{e,f,c,fe,fc,0\}.
\end{equation}
The semigroup $A$ was introduced by its multiplication table and shown to be FB in \cite[Section~ 19]{LZ}.
Its presentation  was recently suggested by Edmond W. H. Lee. 
Let  $\overline{A}$ denote the semigroup anti-isomorphic to $A$.
Zhang and Luo  \cite{ZL} proved that the direct product $A^1 \times \overline{A}^1$ generates another limit variety of monoids.

 Together with Sergey Gusev we calculated that  $A^1 \times  \overline{A}^1$ is equationally equivalent to 
$M_{\gamma}(\{(\mathtt{a^{2+} \circ_\gamma t  \circ_\gamma b^{2+}  \circ_\gamma a^{2+}}), (\mathtt{a^{2+}  \circ_\gamma b^{2+}  \circ_\gamma t  \circ_\gamma a^{2+}})\})$ (Theorem \ref{main1}(iv)).

In Sect.~\ref{sec:new},
using the fact that $(\mathtt{a^{2+}  \circ_\gamma  t   \circ_\gamma b^{2+} \circ_\gamma  a^{2+}})$ and $(\mathtt{a^{2+}  \circ_\gamma b^{2+}  \circ_\gamma t  \circ_\gamma a^{2+}})$ are $\gamma$-terms for $A^1 \times \overline{A}^1$,
 we  give a short proof of  Theorem 6.2 in \cite{WTZ2}  that says that the monoid $A^1 \times \overline{A}^1$ is NFB.

 Let $\mathbb A$  denote the monoid variety generated by $A^1$.  
According to \cite{ZL},  the variety  $\mathbb A \vee \overline{\mathbb A}$ contains 13 finitely generated subvarieties.  In Sect.~\ref{sec:zl} we show that  every subvariety  of $\mathbb A \vee \overline{\mathbb A}$ is generated  by a (finite) monoid of the form $M_\gamma(W)$.

In addition to proving that $\mathbb J$ is a finitely generated limit variety of monoids, Gusev in \cite{SG} calculated the lattice of subvarieties of
$\mathbb J$ which consists of 12 finitely generated varieties. However, the explicit generators of five of these subvarieties including the variety
$\mathbb J$ itself were unknown. In Sect.~\ref{sec:J}   we show that  every subvariety of  $\mathbb J$ is generated  by a (finite) monoid of the form $M_\lambda(W)$.

 In~\cite{GS}, together with Gusev, we  show that the monoids
$M_\lambda(\mathtt{a  \circ_\lambda t  \circ_\lambda b^{2+}  \circ_\lambda a^{2+}})$ and  $M_\rho(\mathtt{a^{2+}  \circ_\rho b^{2+}  \circ_\rho t  \circ_\rho a})$ generate the last two  limit varieties of $J$-trivial monoids.
So, it turns out \cite{GS}, that all limit varieties of $J$-trivial monoids are generated by finite monoids of the form $M_\tau(W)$, where $\tau$ is either the trivial congruence on $\mathfrak A^*$ or $\tau \in \{\gamma, \lambda, \rho\}$.

\section{ Dillworth-Perkins construction in terms of $\tau$-words}

Given a congruence $\tau$ on $\mathfrak A^*$, the `subword" relation  $\le$  between words can be generalized to a relation between $\tau$-words as follows. Given two $\tau$-words ${\mathtt u}, {\mathtt v} \in \mathfrak A^*/\tau$  we write ${\mathtt v} \le_\tau {\mathtt u}$ if ${\mathtt u} = {\mathtt p}\circ_\tau {\mathtt v}\circ_\tau {\mathtt s}$ for some   ${\mathtt p}, {\mathtt s} \in \mathfrak A^*/\tau$.

\begin{lemma} \label{subword} For ${\mathtt u}, {\mathtt v} \in \mathfrak A^*/ \tau$  the following are equivalent:

(i)  ${\mathtt v} \le_\tau {\mathtt u}$;

(ii) every word  ${\bf v} \in {\mathtt v}$  is a subword of a word  ${\bf u} \in {\mathtt u}$;

(iii) some word ${\bf v} \in {\mathtt v}$  is a subword of a word  ${\bf u} \in {\mathtt u}$.

\end{lemma}

 \begin{proof}  (i) $\Rightarrow$ (ii) Since  ${\mathtt v} \le_\tau {\mathtt u}$, we have  ${\mathtt u} = {\mathtt p}\circ_\tau {\mathtt v}\circ_\tau {\mathtt s}$ for some   ${\mathtt p}, {\mathtt s} \in \mathfrak A^*/\tau$.

Take some $\mathbf p \in {\mathtt p}$ and some  $\mathbf s \in {\mathtt s}$. Then $H_\tau(\mathbf{pvs}) =
\mathtt{p} {\circ_\tau}\mathtt{v} {\circ_\tau}\mathtt{s} = {\mathtt u}$. Hence $\bf v$ is a subword of ${\bf u} = \mathbf{pvs} \in {\mathtt u}$.

 Implication (ii) $\Rightarrow$ (iii) is evident.

 (iii) $\Rightarrow$ (i) Since ${\bf v}$ is a subword of ${\bf u}$, we have ${\bf u} = {\bf p v} {\bf s}$ for some ${\bf p}, {\bf s} \in \mathfrak A^*$. Consequently,  ${\mathtt u} = H_\tau ({\bf p})\circ_\tau {\mathtt v} \circ_\tau H_\tau({\bf s})$. Hence  ${\mathtt v} \le_\tau {\mathtt u}$.
\end{proof}

Given a  congruence $\tau$ and a
set of $\tau$-words ${\mathtt W}$, let ${\mathtt W}^{\le_\tau}$ denote the closure of ${\mathtt W}$ in quasi-order $\le_\tau$.

\begin{definition} \label{D: MtauW1}
 Let $\tau$ be a congruence  on the free monoid  $\mathfrak A^*$ and ${\mathtt W} \subseteq \mathfrak A^*/ \tau$ be a set of $\tau$-words. Then $M_\tau({\mathtt W})$ denotes the Rees quotient of $\mathfrak A^*/ \tau$  over the ideal $(\mathfrak A^*/ \tau) \setminus  {\mathtt W}^{\le_\tau}$.
\end{definition}

Lemma~\ref{subword} implies that given a set of words  $W \subseteq \mathfrak A^*$ which is a union of $\tau$-classes,  we have
$H_\tau(\mathfrak A^* \setminus  W^{\le})= (\mathfrak A^*/ \tau) \setminus  {\mathtt W}^{\le_\tau}$, where $\mathtt W = H_\tau(W) \subseteq \mathfrak A^*/\tau$.
Therefore,  Definitions~\ref{D: MtauW} and  \ref{D: MtauW1} are equivalent.

According Definition~\ref{D: MtauW1}, given a set of $\tau$-words ${\mathtt W}$ we have

\begin{equation} \label{mult in MW} M_\tau({\mathtt W}) = \langle {\mathtt W}^{\le_\tau} \cup \{0\}, \star_\tau \rangle,\end{equation}
 where multiplication $\star_\tau$ is given by

$\bullet$ ${\mathtt u} \star_\tau {\mathtt v} = \mathtt {u} \circ_\tau \mathtt{v}$ if   $\mathtt {u} \circ_\tau \mathtt{v} \in {\mathtt W}^{\le_\tau}$, otherwise, $\mathtt {u} \star_\tau \mathtt {v} = 0$.

\begin{prop} \label{P: MW}  Let $\tau$ be a congruence on the free monoid $\mathfrak A^*$ such that the empty word $1$  forms a singleton $\tau$-class.
Let $W \subseteq \mathfrak A^*$ be a set of words closed under $\tau$ and  ${\mathtt W} = H_\tau(W)  \subseteq \mathfrak A^*/ \tau$  be the corresponding set of $\tau$-words. Then for every monoid variety $\vv$  the following are equivalent:

(i)  $\vv$ contains $M_\tau(W) = M_\tau({\mathtt W})$;

(ii)  every word in $W^\le$ is a $\tau$-term for $\vv$;

(iii)  every $\tau$-word in ${\mathtt W}^{\le_\tau}$ is a $\tau$-term for $\vv$.

\end{prop}

\begin{proof}  Since $M_\tau(W) = M_\tau(W^\le)$,     the equivalence of (i) and (ii) is a consequence of  Fact~\ref{prec}.
Since $\mathtt W = H_\tau(W)$,  the equivalence of (ii) and (iii) is a consequence of the definition of a $\tau$-term (see the introduction).
\end{proof}

The  definition of what it means for a word  to be a  $\tau$-term for a monoid variety readily implies the following.

\begin{obs} \label{O: twocongr}  Let $\tau$ and $\tau'$ be two congruences on $\mathfrak A^*$. If   $\tau$ partitions $\tau'$
then every $\tau$-term for a monoid variety $\vv$ is also a  $\tau'$-term for $\vv$.
\end{obs}

 Let  $\mathbb M_\tau(\mathtt W)$ denote the variety generated by  $M_\tau(\mathtt W)$.

\begin{cor}\label{product}  Let  $\sigma_1$ and $\sigma_2$ be congruences on the free monoid  $\mathfrak A^*$  such that the empty word $1$ forms a singleton $\sigma_1$-class and a singleton $\sigma_2$-class.   Let ${W}_1$ and $ {W}_2$ be two sets of words closed under relations $\sigma_1$ and $\sigma_2$ respectively.  Then
\begin{equation} \label{e: product}
\mathbb M_{\sigma_1}(W_1) \vee  \mathbb M_{\sigma_2}(W_2) \subseteq \mathbb M_{\sigma_1 \wedge \sigma_2}(W_1 \cup W_2).
\end{equation}
If $\sigma_1=\sigma_2=\tau$ then
\begin{equation}  \label{e: product1}
\mathbb M_{\tau}(W_1) \vee  \mathbb M_{\tau}(W_2) = \mathbb M_{\tau}(W_1 \cup W_2).
\end{equation}
If $W_1=W_2=W$ then
\begin{equation} \label{e: product2}
\mathbb M_{\sigma_1}(W) \vee  \mathbb M_{\sigma_2}(W) = \mathbb M_{\sigma_1 \wedge \sigma_2}(W).
\end{equation}

\end{cor}

\begin{proof} Since $W_1$ is a union of $\sigma_1$-classes and  $W_2$ is a union of $\sigma_2$-classes, the set $W_1 \cup W_2$ is a union
of  $(\sigma_1 \wedge \sigma_2)$-classes.
Since every word in $W^\le_1$ is a $(\sigma_1 \wedge \sigma_2)$-term for $\mathbb M_{\sigma_1 \wedge \sigma_2}(W_1 \cup W_2)$ by Proposition~\ref{P: MW}, it is a $\sigma_1$-term for  $\mathbb M_{\sigma_1 \wedge \sigma_2}(W_1 \cup W_2)$ by Obsevation~\ref{O: twocongr}.
Hence the variety  $\mathbb M_{\sigma_1 \wedge \sigma_2}(W_1 \cup W_2)$ contains $\mathbb M_{\sigma_1}(W_1)$ by Proposition~\ref{P: MW}.
Similar arguments show that  $\mathbb M_{\sigma_1 \wedge \sigma_2}(W_1 \cup W_2)$ contains $\mathbb M_{\sigma_2}(W_2)$.
Thus \eqref{e: product} holds.

Conversely, since the variety  $\mathbb M_{\sigma_1}(W_1) \vee  \mathbb M_{\sigma_2}(W_2)$ contains both $M_{\sigma_1}( {W}_1)$ and  $M_{\sigma_2}( {W}_2)$, every word in $W^\le_1$ is a $\sigma_1$-term for  $\mathbb M_{\sigma_1}(W_1) \vee  \mathbb M_{\sigma_2}(W_2)$
and  every word in $W^\le_2$ is a $\sigma_2$-term for  $\mathbb M_{\sigma_1}(W_1) \vee  \mathbb M_{\sigma_2}(W_2)$.

If $\sigma_1=\sigma_2=\tau$ then every word in $W^\le_1 \cup W^\le_2 = (W_1 \cup W_2)^\le$ is a $\tau$-term for  $\mathbb M_{\tau}(W_1) \vee  \mathbb M_{\tau}(W_2)$.  Therefore, $\mathbb M_{\tau}(W_1) \vee  \mathbb M_{\tau}(W_2) \supseteq \mathbb M_{\tau}(W_1 \cup W_2)$ by  Proposition~\ref{P: MW}. In view of  \eqref{e: product}, the equality \eqref{e: product1} holds.

If $W_1=W_2=W$ then  every word in $W^\le$  is a $\sigma_1$-term for  $\mathbb M_{\sigma_1}(W) \vee  \mathbb M_{\sigma_2}(W)$ and a
 $\sigma_2$-term for  $\mathbb M_{\sigma_1}(W_1) \vee  \mathbb M_{\sigma_2}(W_2)$. Hence  every word in $W^\le$  is a $(\sigma_1 \wedge \sigma_2)$-term for  $\mathbb M_{\sigma_1}(W) \vee  \mathbb M_{\sigma_2}(W)$. Therefore, $\mathbb M_{\sigma_1}(W) \vee  \mathbb M_{\sigma_2}(W) \supseteq  \mathbb M_{\sigma_1 \wedge \sigma_2}(W)$  by  Proposition~\ref{P: MW}. In view of  \eqref{e: product}, the equality \eqref{e: product2} holds.\end{proof}

If $\tau$ is the trivial congruence then  equation~\eqref{e: product1} turns into
Lemma 5.1 in \cite{JS}.

\section{Congruences $\tau_m$ and $\gamma_k$}

Let  $\con_n({\bf u})$ denote the set of all letters that appear $n$ times in a word $\bf u$. 
For each $k \ge 0$ and  every ${\bf u}, {\bf v} \in \mathfrak A^\ast$   define:

\begin{itemize}
\item ${\bf u} \gamma_{k} {\bf v}$ if and only if  $\con_i({\bf u}) = \con_i( {\bf v})$   for each $0 \le i \le k$.

 In particular, 

\item ${\bf u} \gamma_0 {\bf v}$ if and only if $\con({\bf u}) = \con( {\bf v})$;

\item ${\bf u} \gamma_1 {\bf v}$ if and only if $\simp({\bf u}) = \simp( {\bf v})$ and  $\mul({\bf u}) = \mul( {\bf v})$.

\end{itemize}

It is easy to check that  $\gamma_k$ is a congruence on  $\mathfrak A^\ast$  for each $k \ge 0$. 

 For brevity, if $\mathtt w_1,\mathtt w_2,\dots,\mathtt w_k\in\mathfrak A^\ast/ \tau$, then we write $M_\tau(\mathtt w_1,\mathtt w_2,\dots,\mathtt w_k)$ \\ ($\mathbb M_\tau(\mathtt w_1,\mathtt w_2,\dots,\mathtt w_k)$)
rather than $M_\tau(\{\mathtt w_1,\mathtt w_2,\dots,\mathtt w_k\})$ ($\mathbb M_\tau(\{\mathtt w_1,\mathtt w_2,\dots,\mathtt w_k\})$).
The following observation is a reformulation of a well-known and  easily verified statement.

\begin{obs} \label{O: gammak} 

 For each $k \ge 0$, the variety $\mathbb M(a^k)$ satisfies an identity  ${\bf u} \approx {\bf v}$ if and only if  ${\bf u} \gamma_k {\bf v}$.

\end{obs}

For each $m \ge 0$, let $\tau_{m}$ denote the congruence on ${\mathfrak A}^*$ induced by the relations $a=a^{1+m}$ for every $a \in{\mathfrak A}$. In particular,  $\tau_0$ is the trivial congruence on  ${\mathfrak A}^*$.

In view of  Observation~\ref{O: gammak}, the congruence $\gamma_k$ is fully invariant  for each $k \ge 0$.
In contrast,  if $m \ge 1$ then  $\tau_m$ is not fully invariant.

If ${\mathbf u} \in \mathfrak A^*$ and $x \in \con({\bf u})$ then an {\em island} formed by $x$  in {\bf u} is a maximal subword of $\bf u$ which is a positive power of $x$. For example, the word $xy^2x^5yx^3$ has  three islands formed by $x$ and two islands formed by $y$.
  Clearly, two words $\bf u$ and $\bf v$ are $\tau_1$-related if and only if $\bf v$ can be obtained from $\bf u$ by changing the
exponents of the  {\em corresponding}  islands. For example, ${\bf u} =(xy^2x^5yx^3) \tau_1 (x^2yx^3y^4x^7)={\bf v}$ and the islands $x^5$  in $\bf u$ and $x^3$ in $\bf v$ are  {\em corresponding}.

If a monoid $M$ (monoid variety $\vv$) satisfies all identities in a set $\Sigma$ then we write $M \models \Sigma$ ($\vv \models \Sigma$).

\begin{lemma} \label{L: tau}  For each $m\ge 0$, every subword of a $\tau_m$-term  for a monoid variety $\vv$ is  a $\tau_m$-term for $\vv$.

\end{lemma}

\begin{proof}  It is well-known and can be easily verified that every subword of an isoterm ($\tau_0$-term) for $\vv$ is  an isoterm for $\vv$. 
Similar argument works to show that every subword of a $\tau_1$-term  ${\mathbf u}$ for $\vv$ is  a $\tau_1$-term for $\vv$. Indeed,
let  ${\mathbf v}$ be a subword of  ${\mathbf u}$.
  If $\bf v$ is not a $\tau_1$-term for $\vv$ then ${\vv} \models {\bf v} \approx {\bf w}$ such that ${\bf v}$ and  ${\bf w}$ are not $\tau_1$-related. This means that  ${\bf v} ={\bf p}a{\bf s}$
and  ${\bf w} ={\bf p}'b{\bf s}'$ where  ${\mathbf p}$, ${\mathbf s}$, ${\mathbf p}'$ and ${\mathbf s}'$ are possibly empty words,
${\bf p} \tau_1 {\bf p}'$ and $a \ne b \in \mathfrak A$.
Since  ${\bf v} \le {\bf u}$ we have  ${\bf u} ={\bf c vd}$ for some possibly empty words ${\bf c}$ and $\bf d$. 
Then ${\vv} \models ({\bf u} = {\bf c}{\bf p}a{\bf sd}) \approx   {\bf c}{\bf p}'b{\bf s}'{\bf d}$. Since ${\bf u}$ and  ${\bf c}{\bf p}'b{\bf s}'{\bf d}$ are not $\tau_1$-related, we obtained a contradiction to the fact that
$\bf u$ is a $\tau_1$-term for $\vv$.  Therefore, $\bf v$  must be a $\tau_1$-term for $\vv$.

 If $m >1$ and $\bf u$  is a $\tau_m$-term for a monoid variety $\vv$, then $\bf u$  is a $\tau_1$-term for $\vv$ by Observation~\ref{O: twocongr}. It follows that every subword  ${\mathbf v}$  of  ${\mathbf u}$ is a  $\tau_1$-term for $\vv$. If  ${\mathbf v}$ is not a  $\tau_m$-term for $\vv$ then  ${\vv} \models {\bf v} \approx {\bf w}$ such that ${\bf v}$ and  ${\bf w}$ are not $\tau_m$-related. Since ${\bf v} \tau_1 {\bf w}$ but ${\bf v} \not \tau_m {\bf w}$,  for some letter $x \in \mathfrak A$, the corresponding islands $x^p$ and $x^q$ in $\bf v$ and $\bf w$ are such that $p \not \equiv q  \pmod m$.  Since  ${\bf v} \le {\bf u}$ we have  ${\bf u} ={\bf c vd}$ for some possibly empty words ${\bf c}$ and $\bf d$. 
Then ${\vv} \models {\bf u}  \approx   {\bf c}{\bf w}{\bf d}$. 
Since the islands $x^p$ and $x^q$ in $\bf v$ and $\bf w$ are corresponding in  ${\bf u}$ and  ${\bf c}{\bf w}{\bf d}$ as well, the words
 ${\bf u}$ and  ${\bf c}{\bf w}{\bf d}$ are not $\tau_m$-related. This is impossible, because
$\bf u$ is a $\tau_m$-term for $\vv$. To avoid a contradiction, we conclude that $\bf v$ is also a $\tau_m$-term for $\vv$.
\end{proof}

\begin{lemma} \label{twoletters}  Let $m, k \ge 0$ and $\vv$ be a monoid variety such that $x^k$ is an isoterm for $\vv$. Then every subword of a $(\tau_m \wedge \gamma_k)$-term for  $\vv$ is  a  $(\tau_m \wedge \gamma_k)$-term for $\vv$.

\end{lemma}

\begin{proof} Since ${\bf u} \gamma_{0} {\bf v}$ if and only if  $\con({\bf u}) = \con( {\bf v})$, the congruence
$\tau_m$ partitions $\gamma_0$ for each $m \ge 0$. Hence  $\tau_m = \tau_m \wedge \gamma_0$ and the statement follows from
Lemma~\ref{L: tau}. So, we may assume that $k \ge 1$.

 Suppose that ${\bf v} \le {\bf u}$ and $\bf u$ is a $(\tau_m \wedge \gamma_k)$-term for $\vv$. In view of Observation~\ref{O: twocongr} and Lemma~\ref{L: tau}, the word  $\bf v$ is a  $\tau_m$-term for $\vv$. Since  $x^k$ is an isoterm for $\vv$,  the word  $\bf v$ is a  $(\tau_m\wedge \gamma_k)$-term for $\vv$ by Observation~\ref{O: gammak}.
\end{proof}

Notice that  $\gamma = \tau_1 \wedge \gamma_1$, where $\gamma$ is the equivalence relation defined in the introduction.

\begin{fact} \label{F: gamma} Let $\mathbf  u$ be a word that involves at least two distinct letters.
If $\mathbf u$ is a $\gamma$-term for a monoid variety $\vv$ then $x$ is an isoterm for $\vv$.
\end{fact}

\begin{proof} Since  $\mathbf u$ is a $\gamma$-term for $\vv$, the word $\mathbf u$ is also a $\tau_1$-term for $\vv$ by  Observation~\ref{O: twocongr}.
 If $x$ is not  an isoterm for $\vv$ then $\vv \models x \approx x^p$ for some $p \ge 2$.
But then  $\vv \models {\bf u} \approx {\bf u}^p$ which contradicts the fact
that  ${\mathbf u}$ is a $\tau_1$-term for  ${\vv}$.
Therefore,  $x$ must be an isoterm for $\vv$.
\end{proof}

\begin{cor} \label{C: taum}  Let $m \ge 0$, $\tau \in \{\tau_m, \gamma\}$ and  ${\mathtt v} \le_{\tau} {\mathtt u}$. If $\mathtt u$ is a $\tau$-term  for a monoid variety $\vv$ then ${\mathtt v}$ is also a $\tau$-term for $\vv$.

\end{cor}

\begin{proof} Since $\mathtt u$  is a $\tau$-term for $\vv$,  every word in $\mathtt u$ is a $\tau$-term for $\vv$. Since 
${\mathtt v} \le_{\tau} {\mathtt u}$, every word ${\bf v} \in \mathtt v$ is a subword of some word  ${\bf u} \in \mathtt u$ by Lemma~\ref{subword}.

 If $\tau = \tau_m$, then  
$\bf v$ is a $\tau_m$-term for $\vv$ by Lemma~\ref{L: tau}.

So, we may assume that $\tau=\gamma$.  If $\bf u$ involves at least two distinct letters, then $\bf v$ is a $\gamma$-term for $\vv$ by Lemma~\ref{twoletters} and Fact~\ref{F: gamma}.
Thus we may assume that $\mathtt u$ involves at most one letter $x$.

If $\mathtt u = \{1\}$ then $\mathtt{v} = \{1\}$  is a $\gamma$-term for $\vv$.

If $\mathtt{u} = \{x\}$  then $x$ is an isoterm for $\vv$.  Hence $\mathbf v$ is a $\gamma$-term for $\vv$ by Lemma~\ref{twoletters}.

It is left to assume that  $\mathtt{u} =\{x^2, x^3, x^4, \dots\}$  for some $x \in \mathfrak A$. Since no word in this $\gamma$-class can form an identity of $\vv$ with $x$, the word $x$ is an isoterm for $\vv$ in this case as well.  Hence $\mathbf v$ is a $\gamma$-term for $\vv$ by Lemma~\ref{twoletters}.
\end{proof}

\begin{cor} \label{C: MW1}  Let $m \ge 0$, $\tau \in \{\tau_m, \gamma\}$ and $\mathtt{W} \subseteq \mathfrak A^*/\tau$ be a set of $\tau$-words. Then a monoid variety $\vv$ contains  $M_\tau(\mathtt {W})$ if and only if every $\tau$-word in $\mathtt W$ is a $\tau$-term for $\vv$.
\end{cor}

\begin{proof}  If $\vv$ contains  $M_\tau(\mathtt W)$ then every $\tau$-word in ${\mathtt W}^{\le_\tau}$ is a $\tau$-term for $\vv$ by Proposition~\ref{P: MW}.
In particular,  every $\tau$-word in $\mathtt W$ is a $\tau$-term for $\vv$. 

Conversely, suppose that  every $\tau$-word in $\mathtt W$ is a $\tau$-term for $\vv$.
If ${\mathtt  v} \in \mathtt W^{\le_\tau}$  then ${\mathtt v} \le_\tau {\mathtt u}$ for some $\mathtt u \in \mathtt W$.
Consequently,  $\mathtt v$  is a $\tau$-term for $\vv$  by Corollary~\ref{C: taum}.
Therefore,  $\vv$ contains  $M_\tau(\mathtt {W})$ by Proposition~\ref{P: MW}.
\end{proof}

Notice that if $\tau = \tau_0$ then Corollary~\ref{C: MW1} coincides with Lemma 3.3 in \cite{MJ}.

\section{Every subvariety of $\mathbb A \vee \overline{\mathbb A}$ is generated by a monoid of the form $M_{\gamma}(W)$}  \label{sec:zl}

In consistence with notations used in the introduction,  define  
\begin{itemize}
\item $\mathtt {a} = \{a\}$;
\item $\mathtt {a}^{n+} = \{a^n, a^{n+1},  a^{n+2}, \dots\}$ for each $n \ge 1$ and $a \in \mathfrak A$.
\end{itemize}

We often simplify the following notations: 
\begin{itemize}

\item  instead of  $\mathtt {a}^{2+}$ we write  $\mathtt {a}^{+}$, that is,  $\mathtt {a}^{+} = \mathtt {a}^{2+} =\{a^2, a^{3},  \dots\}$;

\item  when the congruence $\tau$ is clear from the context, we omit the sign  of operation  ''$\circ_\tau$''  in  $\mathfrak A^*/\tau$, that is, 
we write $\mathtt w_1 \mathtt w_2$ instead of  $\mathtt w_1 \circ_\tau \mathtt w_2$. 
\end{itemize}

 If $M$ and $M'$ are isomorphic monoids we write $M \cong M'$. 
 Figure~\ref{pic: A} below duplicates  Figure~1 in \cite{ZL} which exhibits the subvariety lattice of  $\mathbb A \vee \overline{\mathbb A}$.
The four-element chain at the bottom of this lattice is the same as the corresponding chain in Figure 1 in \cite{MJ}:
\[\mathbb M(\emptyset) \subset \mathbb M(1) \subset \mathbb M(a)  \subset \mathbb M(ab).\]

Note that the congruences and sets of words can be chosen in multiple ways to generate the same variety.  For example, 
 $\mathbb M(\emptyset) = \mathbb M_\tau (\emptyset)$ and  $\mathbb M(1) = \mathbb  M_{\tau}(\{1\}) =\mathbb M_{\sigma} (\mathtt{a}^{1+})$, where
 $\tau$ is any congruence on $\mathfrak A^*$ such that the empty word $1$ forms a singleton $\tau$-class
and $\sigma$  is any congruence  such that $\mathtt{a}^{1+}$ is a $\sigma$-class.
 Notice that $M(1) \cong  M_{\tau}(\{1\}) = \{0,1\}$ is the 2-element semilattice while
$M_{\sigma}(\mathtt{a}^{1+}) = \{1, {\mathtt a}^{1+}, 0\}$ is  the 3-element semilattice.

Another example,  $\mathbb M(a) = \mathbb M_{\delta} \mathtt{(a}) = \mathbb M_{\gamma} (\mathtt {a}^{2+})$ and 
$\mathbb M(ab) = \mathbb M_{\delta} (\mathtt{ab})$, where $\delta$ is  any congruence such that each letter $a \in \mathfrak A$ forms a singleton $\delta$-class. 
Notice that $M(a) \cong  M_{\gamma} \mathtt{(a}) = \{1,a, 0\}$ is a cyclic monoid while  $M_{\gamma} (\mathtt {a}^{2+}) = \{1, \mathtt{a}, \mathtt {a}^{2+}, 0\}$ is a cyclic monoid with zero adjoint.

Recall from the introduction that  \eqref{A01} is a presentation of  the 5-element monoid $A_0^1$.
Let $E^1$ and  $B_0^1$ denote  the monoids  obtained  by adjoining an identity element to the following  semigroups:
\[E = \langle e, c \mid  e^2=e, c^2=ec=0, ce=c \rangle;\]
\[B_0 = \langle e, f, c \mid  e^2=e, f^2 =f, ef=fe=0, ec=cf=c  \rangle.\]
The four-element monoid $E^1$ and the five-element monoid $B_0^1$ were shown to be FB in  \cite{1977}.
If $M$ is a monoid then  $\var M$ denotes the variety generated by $M$.
Denote  $\var A^1_0$ by $\mathbb A_0$,   $\var E^1$ by $\mathbb E$  and  $\var B_0^1$ by $\mathbb B_0$.

The subvariety  lattice of $\mathbb A \vee \overline{\mathbb A}$ in Figure~\ref{pic: A} contains  the variety $\mathbb A_0$, whose lattice of subvarieties was first described in  \cite{ELB0}. The following theorem assigns a generating monoid  of the form $M_{\gamma}({\mathtt W})$
 to each variety in the interval from $\mathbb D= \mathbb M(ab)$ to $\mathbb A_0$.

\begin{theorem} \label{latticeA}

(i) ${\mathbb E} = \mathbb M_\gamma(\mathtt{ta^+})$.

(ii) $\overline{\mathbb E} = \mathbb M_\gamma(\mathtt{a^+t})$.

(iii) ${\mathbb E} \vee \overline{\mathbb E} \stackrel{\cite{ELB0}}{=} \mathbb B_0 = \mathbb M_{\gamma}(\mathtt{ta^+, a^+t})  = \mathbb M_{\gamma}(\mathtt{a^+ta^+})$.

(iv) $\mathbb A_0 = \mathbb M_\gamma(\mathtt{a^+b^+})$.

\end{theorem}

\begin{proof} (i)  Let us verify that every word in $(\mathtt{t \circ_\gamma x^+}) = \{tx^n \mid n \ge 2\}$ is a $\gamma$-term for $E^1$. Indeed, suppose that for some $n\ge 2$, the monoid $E^1$ satisfies an identity $tx^n \approx {\bf u}$ such that $tx^n$ and $\bf u$ are not $\gamma$-related. Since $E^1 \models x^2 \approx x^3$  and $xy$ is an isoterm for $E^1$, we conclude that $E^1$ satisfies one of the following identities:
\[tx^2 \approx x^2t,  tx^2 \approx xtx^2, tx^2 \approx x^2tx^2\]
If we substitute $c$ for $t$ and $e$ for $x$, each of these identities yields $c = 0$.  To avoid a contradiction, we conclude that
$\mathtt{ta^+}$ is a $\gamma$-term for $E^1$.  Hence the variety ${\mathbb E}$ contains $M_{\gamma}(\mathtt{ta}^+)$ by 
Corollary~\ref{C: MW1}.

 Since ${\mathbb D} \models tx^2 \approx x^2t$, the word $tx^2$
is not a $\gamma$-term for $\mathbb D$. Therefore,
the variety $\mathbb D$ does not contain $M_{\gamma}(\mathtt{ta^+})$ by Corollary~\ref{C: MW1}.
Since $\mathbb D$ is a unique maximal subvariety of $\mathbb E$, the variety $\mathbb E$ is generated by $M_{\gamma}(\mathtt{ta^+})$.

Part (ii) is dual to Part (i).

(iii) In view of Parts (i)--(ii) and Corollary \ref{product}, we have
 ${\mathbb E} \vee \overline{\mathbb E} = \mathbb M_{\gamma}(\mathtt{ta^+, a^+t})$.

Let us verify that $\mathbb M_{\gamma}(\mathtt{ta^+, a^+t})  = \mathbb M_{\gamma}(\mathtt{a^+ta^+}).$

Indeed,   in view of Corollary~\ref{C: MW1},  the $\gamma$-word $(\mathtt{a^+ \circ_\gamma t})$ is a $\gamma$-term for  $\vv= \mathbb M (\mathtt{a^+t, ta^+})$.  Since $\mathtt{t}  \le_\gamma (\mathtt{t  \circ_\gamma a^+})$, the  $\gamma$-class ${\mathtt t}$ is a $\gamma$-term for $\vv$ by  Corollary~\ref{C: taum}.  Since ${\mathtt t} = \{t\}$, the word $t$ is an isoterm for $\vv$.

If  $(\mathtt{a^+ \circ_\gamma t \circ_\gamma a^+})$ is not a $\gamma$-term for $\vv$ then  $\vv \models x^n t x^m \approx {\bf u}$ for some $n,m \ge 1$ and
some word ${\mathbf u}$ that  is not $\gamma$-related to $x^n t x^m$.
Observation~\ref{O: gammak}  implies that for some $k \ge 2$, either $\vv \models  x^n t x^m \approx x^k t$
or $\vv \models  x^n t x^m \approx  t x^k$. But this is impossible because both $x^kt$ and $tx^k$ are $\gamma$-terms for  $\vv$ by Corollary~\ref{C: MW1}. Therefore,  $(\mathtt{a^+ \circ_\gamma t \circ_\gamma a^+})$ is a $\gamma$-term for  $\vv=\mathbb M_\gamma (\mathtt{a^+t, ta^+})$ and consequently, $\mathbb M_\gamma (\mathtt{a^+t, ta^+}) \supseteq \mathbb M_ \gamma (\mathtt{a^+ta^+})$
by Corollary~\ref{C: MW1}.

Conversely, since $(\mathtt{a^+ \circ_\gamma t })  \le_\gamma (\mathtt{a^+ \circ_\gamma t  \circ_\gamma a^+})$ and 
 $(\mathtt{t \circ_\gamma a^+ })  \le_\gamma (\mathtt{a^+ \circ_\gamma t  \circ_\gamma a^+})$, both  $(\mathtt{a^+ \circ_\gamma t })$ and
 $(\mathtt{t \circ_\gamma a^+ })$ are  $\gamma$-terms for $\mathbb M_ \gamma (\mathtt{a^+ta^+})$ by Corollary~\ref{C: taum}.
Therefore,  $\mathbb M_\gamma (\mathtt{a^+t, ta^+}) \subseteq \mathbb M_ \gamma (\mathtt{a^+ta^+})$
by Corollary~\ref{C: MW1}

(iv) As verified in \cite{OS}, the $\tau_1$-word  $(\mathtt{a^{1+} \circ_{\tau_1} b^{1+}}) = \{a^nb^m \mid n,m \ge 1\}$ is a $\tau_1$-term for  $A_0^1$. Since  the $\gamma$-class $(\mathtt{a^{2+} \circ_\gamma b^{2+}})$ is a subset of $(\mathtt{a^{1+} \circ_{\tau_1} b^{1+}})$, every word in $(\mathtt{a^{2+} \circ_\gamma b^{2+}}) = \{a^nb^m \mid n,m \ge 2\}$ is a $\tau_1$-term for $A_0^1$.
If some word in $(\mathtt{a^+ \circ_\gamma b^+}) = (\mathtt{a^{2+} \circ_\gamma b^{2+}})$ is not a $\gamma$-term for $A_0^1$,  then  $A_0^1 \models x^ny^m  \approx x^ky$ or $A_0^1 \models x^ny^m  \approx xy^k$ for some $n, m \ge 2$ and $k \ge 1$. Since each of these identities fails in 
$A_0^1$, the $\gamma$-word  $(\mathtt{a^+ \circ_\gamma b^+})$ is a $\gamma$-term for $A_0^1$.
Hence the variety $\mathbb A_0$ contains  $M_\gamma(\mathtt{a^+b^+})$ by Corollary~\ref{C: MW1}.

On the other hand, since $B_0^1 \models x^2y^2 \approx y^2x^2$,  the word $x^2y^2$ is not a  $\gamma$-term for $B_0^1$.  Consequently, $\mathbb B_0$ does not contain  $M_\gamma(\mathtt{a^+b^+})$ by Corollary~\ref{C: MW1}.
 Since   $\mathbb B_0$ is a unique maximal subvariety of $\mathbb A_0$, we have $\mathbb A_0 = \mathbb M_\gamma(\mathtt{a^+b^+})$.
\end{proof}

\begin{remark} (i)  The monoid $E^1$ is isomorphic to the submonoid $\{ 1,  \mathtt{a^+, ta^+} , 0\}$ of  $M_\gamma(\mathtt{ta^+})$.

(ii) The five-element monoid $B_0^1$ can be obtained from the submonoid  \\ $\{1, \mathtt{ta^+, a^+, b^+t, b^+}, 0\}$ of
$M_{\gamma}(\mathtt{ta^+, b^+t})$ by identifying $c = \mathtt{ta^+ = b^+t}$.

(iii) $A_0^1 \stackrel{\cite{OS}}{\cong}  M_{\tau_1}(\mathtt{a^{1+}b^{1+}})$  is isomorphic  to the  submonoid \\ $\{1, \mathtt{a^+, b^+, a^+b^+}, 0\}$ of $M_\gamma(\mathtt{a^+b^+})$.
\end{remark}

\begin{figure}[htb]
\unitlength=1mm
\linethickness{0.4pt}
\begin{center}
\begin{picture}(55,98)
\put(35,5){\circle*{1.33}}
\put(35,15){\circle*{1.33}}
\put(35,25){\circle*{1.33}}
\put(35,35){\circle*{1.33}}
\put(45,45){\circle*{1.33}}
\put(25,45){\circle*{1.33}}
\put(35,55){\circle*{1.33}}
\put(25,65){\circle*{1.33}}
\put(45,65){\circle*{1.33}}
\put(35,75){\circle*{1.33}}
\put(45,85){\circle*{1.33}}
\put(25,85){\circle*{1.33}}
\put(35,95){\circle*{1.33}}

\put(35,5){\line(0,1){30}}
\put(35,35){\line(-1,1){10}}
\put(35,35){\line(1,1){10}}
\put(25,45){\line(1,1){10}}
\put(25,45){\line(1,1){20}}
\put(25,65){\line(1,1){20}}
\put(45,45){\line(-1,1){20}}
\put(45,65){\line(-1,1){20}}

\put(45,85){\line(-1,1){10}}
\put(25,85){\line(1,1){10}}
\put(35,2){\makebox(0,0)[cc]{${\mathbb M}(\varnothing)$}}
\put(37,15){\makebox(0,0)[lc]{${\mathbb M}(1)$}}
\put(37,25){\makebox(0,0)[lc]{${\mathbb M}(a)$}}
\put(37,35){\makebox(0,0)[lc]{${\mathbb M}(ab)$}}
\put(23,45){\makebox(0,0)[rc]{${\mathbb M}_\gamma(\mathtt{ta^+})$}}
\put(24,85){\makebox(0,0)[rc]{${\mathbb M}_\gamma(\mathtt{a^+b^+ta^+})$}}
\put(72,85){\makebox(0,0)[rc]{${\mathbb M}_\gamma(\mathtt{a^+t b^+a^+})$}}
\put(72,75){\makebox(0,0)[rc]{${\mathbb M}_\gamma(\mathtt{a^+b^+, a^+tsa^+})$}}
\put(68,55){\makebox(0,0)[rc]{$\mathbb B_0 = {\mathbb M}_\gamma(\mathtt{a^+ta^+})$}}
\put(23,65){\makebox(0,0)[rc]{$\mathbb A_0 = {\mathbb M}_{\gamma}(\mathtt{a^+b^+})$}}
\put(47,45){\makebox(0,0)[lc]{${\mathbb M_\gamma(\mathtt{a^+t})}$}}
\put(47,65){\makebox(0,0)[lc]{$ \mathbb Q = {\mathbb M}_{\gamma}(\mathtt{a^+tsa^+})$}}

\put(20, 98){\makebox(0,0)[lc]{$\mathbb A \vee \overline{\mathbb A} = {\mathbb M}_\gamma(\mathtt{a^+b^+ta^+, a^+t b^+a^+})$}}

\end{picture}
\end{center}
\caption{The subvariety  lattice of $\mathbb A \vee \overline{\mathbb A}$ (cf. Figure~1 in \cite{ZL})}
\label{pic: A}
\end{figure}
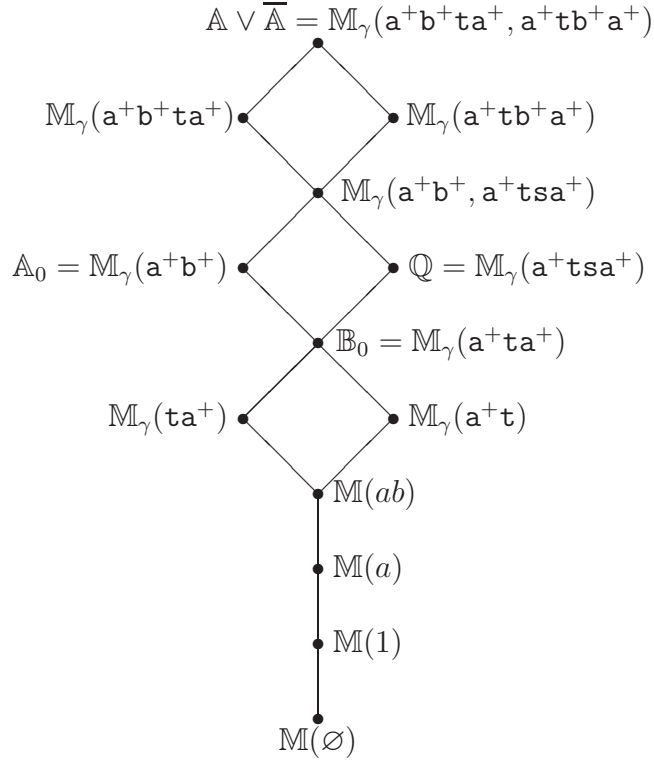

 Let $Q^1$  denote  the monoid  obtained  by adjoining an identity element to the following  semigroup:
\[Q = \langle e, b, c \mid  e^2 = e, eb = b, ce = c, ec = be = cb = 0 \rangle.\]
The semigroup $Q$ was introduced and shown to be FB  in \cite[Section 6.5]{JA}. The finite basis for the six-element monoid $Q^1$ was found in \cite[Section 4.1]{LeeLi}.

 Denote $\var Q^1$ by $\mathbb Q$. 
According to  Figure~\ref{pic: A}, the variety  $\mathbb B_0$ has two covers:   $\mathbb A_0$  and  $\mathbb Q$.
 The following theorem assigns a generating monoid  of the form $M_{\gamma}({\mathtt W})$
 to each variety in the interval from $\mathbb Q$ to $\mathbb A \vee \overline{\mathbb A}$.

\begin{theorem} \label{main1}

(i) $\mathbb Q = \mathbb M_\gamma(\mathtt{a^+ t s a^+})$.

(ii) $\mathbb A \wedge \overline{\mathbb A} \stackrel{\cite{ZL}}{=} \mathbb A_0 \vee \mathbb Q = \mathbb M_\gamma(\mathtt{a^+b^+, a^+ t s a^+})$.

(iii) $\mathbb A = \mathbb M_\gamma(\mathtt{a^+b^+ta^+})$, dually, $\overline{\mathbb A} = \mathbb M_\gamma(\mathtt{a^+tb^+a^+})$.

(iv) $\mathbb A \vee \overline{\mathbb A} = \mathbb M_\gamma (\mathtt{a^+b^+ta^+, a^+tb^+a^+})$.

\end{theorem}

\begin{proof} (i) Suppose that $Q^1 \models x^ntsx^k \approx {\bf u}$  for some $n,k \ge 1$. Since $(\mathtt{a^+ \circ_\gamma t \circ_\gamma a^+})$ is a $\gamma$-term for $B_0^1$ by Theorem~\ref{latticeA}, either  ${\bf u} = x^{p_1}tsx^{p_2}$
or ${\bf u} = x^{q_1}tx^{q_2}sx^{q_3}$ for some $p_1, p_2, q_1, q_2, q_3 \ge 1$. To rule out the second possibility, substitute  $\Theta(x)=e$, $\Theta(t)=b$ and $\Theta(s)=c$. Then
$\Theta( x^{p_1}tsx^{p_2}) = bc$ but $\Theta( x^{q_1}tx^{q_2}sx^{q_3})=0$.
We conclude that $(\mathtt{a^+  \circ_\gamma t  \circ_\gamma s  \circ_\gamma a^+}) = \{a^n  t  s  a^k \mid n, k \ge 1\}$ is a $\gamma$-term for $Q^1$.  Therefore, $\mathbb Q$ contains $M_\gamma(\mathtt{a^+tsa^+})$  by Corollary~\ref{C: MW1}.

 Since $B_0^1 \models xtsx \approx xtxsx$  \cite{1977}, the $\gamma$-word $(\mathtt{a^+  \circ_\gamma t  \circ_\gamma s  \circ_\gamma a^+})$ is not a $\gamma$-term for $B_0^1$.
 Corollary~\ref{C: MW1} implies that  $\mathbb B_0$ does not contain $M_\gamma(\mathtt{a^+tsa^+})$.
Since $\mathbb B_0$ is a unique maximal subvariety of $\mathbb Q$, we have $\mathbb Q = \mathbb M_\gamma(\mathtt{a^+tsa^+})$.

Part (ii) follows from Theorem~\ref{latticeA}(iv), Part~(i) and
Corollary \ref{product}.

(iii)  Recall from the introduction that $A^1$ is obtained by adjoining an identity element to the semigroup $A$
given by presentation \eqref{A}.

Suppose that $A^1 \models x^ny^mtx^k \approx {\bf u}$ for some $n,k \ge 1$ and $m \ge 2$.
 Since the variety $\mathbb A$ contains  $\mathbb B_0= \mathbb M_{\gamma}(\mathtt{a^+ta^+})$,  the $\gamma$-words $(\mathtt{a^+ \circ_\gamma t})$ and $(\mathtt{a^+ \circ_\gamma t \circ_\gamma a^+})$ are $\gamma$-terms
for $A^1$ by Corollaries \ref{C: taum} and \ref{C: MW1}.
 Consequently, ${\bf u} = {\bf a}tx^p$ for some  $p\ge 1$ and ${\bf a} \in \mathfrak A^*$ such that $x$ occurs at least once in $\bf a$ and $y$ occurs at least twice in $\bf a$.

If ${\bf a}$ contains $yx$ as a subword then substituting $\Theta(x)=f$, $\Theta(y)=e$ and
$\Theta(t) =c$ we obtain $\Theta (x^ny^mtx^k) =fc$ but $\Theta({\bf u}) =0$. To avoid a contradiction, we conclude that ${\bf u}=x^qy^rtx^p$ for some $q, p \ge 1$ and $r \ge 2$.
Therefore,  $(\mathtt{a^+ \circ_\gamma b^+ \circ_\gamma t \circ_\gamma a^+})$ is a $\gamma$-term for $A^1$. 
Hence  the variety $\mathbb A$ contains $M_\gamma(\mathtt{a^+b^+ta^+})$ by Corollary~\ref{C: MW1}.

On the other hand,  the identity bases for $A_0^1$\cite{1977} and for $Q^1$\cite{LeeLi} imply that both $A_0^1$ and 
$Q^1$ satisfy $xy^2tx \approx yxytx$.
Hence  $(\mathtt{a^+ \circ_\gamma b^+ \circ_\gamma t \circ_\gamma a^+})$ is not a $\gamma$-term for $A_0^1 \times Q^1$.
Consequently,  $\mathbb A_0 \vee \mathbb Q$  does not contain $M_\gamma(\mathtt{a^+b^+ta^+})$ by  Corollary~\ref{C: MW1}.
Since $\mathbb A_0 \vee \mathbb Q$ is a unique maximal subvariety of $\mathbb A$,  we have $\mathbb A = \mathbb M_\gamma(\mathtt{a^+b^+ta^+})$.

Part (iv) follows from Part~(iii) and Corollary \ref{product}.
\end{proof}

\begin{remark}

(i) The monoid $Q^1$ is isomorphic to the submonoid  
$$\{1, \mathtt{a^+, a^+t, sa^+, a^+tsa^+}, 0\}$$
of $M_\gamma(\mathtt{a^+ t s a^+})$ by $e \rightarrow \mathtt{a^+}$,  $b \rightarrow \mathtt{a^+t}$, $c \rightarrow \mathtt{sa^+}$.

(ii) Gusev noticed that  $A^1$ is isomorphic to the submonoid 
$$\{1, \mathtt{a^+, b^+, a^+b^+, b^+ta^+, a^+b^+ta^+}, 0\}$$
of  $M_\gamma(\mathtt{a^+b^+ta^+})$ by $f \rightarrow \mathtt{a^+}$,  $e \rightarrow \mathtt{b^+}$ and
 $c \rightarrow \mathtt{b^+ta^+}$.

\end{remark}

\section{A short proof that  limit variety  $\mathbb A \vee \overline{\mathbb A}$ is NFB} \label{sec:new}

The goal of this section is to reprove Theorem 6.2 in \cite{WTZ2}  that says that the monoid $A^1 \times \overline{A}^1$ is NFB.
To this aim, we establish two sufficient conditions under which a monoid is NFB. 
The following lemma is an  immediate consequence from Lemma 4.1 in \cite{OS}.

\begin{lemma} \label{redef}    Let ${\bf u}, {\bf v} \in \mathfrak A^+$  be two words such that ${\bf u} \gamma {\bf v}$.  Let $\Theta: \mathfrak A \rightarrow \mathfrak A^+$ be a substitution with the following  property:

($\ast$)  If $x \in \mul({\bf u})$ then $\Theta(x) = y^m$ for some $m \ge 1$ and $y \in \mathfrak A$. 

Then $\Theta({\bf u}) \gamma \Theta({\bf v})$.

\end{lemma}

The next statement readily follows from  Fact 4.2 in \cite{OS}.

\begin{fact} \label{Eu}   
Given a word $\bf u$ and a substitution $\Theta: \mathfrak A \rightarrow \mathfrak A^+$, one can rename some letters in $\bf u$ so that the resulting word $E({\bf u})$ has the following properties:

(i)  $\con (E({\bf u})) \subseteq  \con ({\bf u})$;

(ii) $\Theta(E({\bf u})) \gamma \Theta ({\bf u})$;

(iii) for every  $x, y \in \con (E({\bf u}))$,  if the words $\Theta(x)$  and  $\Theta(y)$ are powers of the same letter then $x=y$.

\end{fact}

\begin{lemma} \cite[Lemma 5.1]{OS} \label{nfblemma}
Let $\tau$ be an equivalence relation on $\mathfrak A^+$ and $S$ be a semigroup.
Suppose that, for infinitely many $n$, $S$ satisfies an identity ${\bf U}_n \approx {\bf V}_n$ in at least $n$ letters
such that ${\bf U}_n$ and ${\bf V}_n$ are not $\tau$-related.

Suppose also that for every identity ${\bf u} \approx {\bf v}$ of $S$ in fewer than $n$ letters, every
 word  $\bf U$ with ${\bf U} \tau {\bf U}_n$ and every substitution
 $\Theta: \mathfrak A \rightarrow \mathfrak A^+$ such that $\Theta({\bf u}) = {\bf U}$, we have
 ${\bf U} \tau \Theta({\bf v})$.  Then $S$ is NFB.

\end{lemma}

We say that a word $\bf u$ is  $xy$-simple if  for each $x\ne y \in \con({\bf u})$,  $xy$ appears at most once in $\bf u$ as a subword.

\begin{sufcon}  \label{SC}  Let  $S$ be a semigroup.  Suppose that for each $n$ big enough one can find an  $xy$-simple  word ${\mathbf U}_n$ in at least $n$ letters such that ${\mathbf U}_n$ is not a $\gamma$-term for $S$.   Suppose also
that every word $\bf u$ involving less than, say, $n$ letters,  is a $\gamma$-term for $S$ whenever  
 $\Theta({\bf u}) = {\bf U}$ for any ${\bf U} \gamma {\bf U}_n$  
  and substitution $\Theta: \mathfrak A \rightarrow \mathfrak A^+$  such that:

($\star$)  for every  $x, y \in \con ({\bf u})$,  if  $\Theta(x)$  and  $\Theta(y)$ are powers of the same letter then $x=y$.

 Then $S$ is NFB.

\end{sufcon}

\begin{proof} Let $n$ be big enough. Since  ${\mathbf U}_n$ is not a $\gamma$-term for $S$, the semigroup  $S$ satisfies an identity
${\bf U}_n \approx {\bf V}_n$ in at least $n$ letters
such that ${\bf U}_n$ and ${\bf V}_n$ are not $\gamma$-related.

Let ${\bf u} \approx {\bf v}$ be an identity of $S$ in fewer than $n$ letters. Let $\bf U$ be a word such that ${\bf U} \gamma {\bf U}_n$ and $\Theta: \mathfrak A \rightarrow \mathfrak A^+$ be a substitution such that $\Theta({\bf u}) = {\bf U}$.
Let  $E({\bf u})$ denote the word obtained from $\bf u$ by renaming some letters so that  $E({\bf u})$ satisfies Properties (i)--(iii) in Fact~\ref{Eu}, that is:

(i) $|\con(E({\bf u}))|  \le |\con({\bf u})| < n$;  

(ii) $\Theta(E({\bf u})) \gamma {\bf U}$ ;

(iii) for each  $x, y \in \con (E({\bf u}))$,  if  $\Theta(x)$  and  $\Theta(y)$ are powers of the same letter then $x=y$.

 Then $E({\bf u})$ is a $\gamma$-term for $S$ by our assumptions.  Hence $E({\bf u}) \gamma E({\bf v})$.

Since ${\mathbf U}_n$ is $xy$-simple and   $\Theta(E({\bf u})) \tau_1 {\bf U}_n$, the word  $\Theta(E({\bf u}))$ is  also $xy$-simple. Hence for each $x \in \mul(E({\bf u}))$, the value of $\Theta(x)$ is a power of a letter.  Thus  Property ($\ast$) in Lemma~\ref{redef} holds and we have
\[ {\bf U} = \Theta({\bf u}) \stackrel{Fact~\ref{Eu}}{\gamma} \Theta (E ({\bf u})) \stackrel{Lemma~\ref{redef}} {\gamma} \Theta (E ({\bf v})) \stackrel{Fact~\ref{Eu}}{\gamma} \Theta({\bf v}).\]
Therefore, $S$ is NFB by Lemma \ref{nfblemma}.
\end{proof}

If letter $x$ forms two islands in a word $\bf u$, we denote these islands by   ${_1x^+}$ and ${_2x^+}$. For example, if ${\bf u} = y^2xz^5x^3y$ then  ${_1x^+} = x$ and ${_2x^+} = x^3$.

\begin{sufcon}
\label{SC: for A}
Let $M$ be a monoid such that for each $n>1$ the word 
${\mathbf U}_n=x y_1^2y_2^2\cdots y^2_{n-1} y_n^2x$ is not a $\gamma$-term for $M$, but   $(\mathtt{a^+ \circ_\gamma t  \circ_\gamma b^+ \circ_\gamma a^+})$ and  \\ $(\mathtt{a^+   \circ_\gamma b^+ \circ_\gamma t \circ_\gamma a^+})$ 
 are $\gamma$-terms for $M$. Then $M$ is NFB.
\end{sufcon}

\begin{proof} 
 Let $\bf u$ be a word with $|\con({\bf u})| < n$ and $\Theta: \mathfrak A \rightarrow \mathfrak A^+$ be a substitution  such that $\Theta({\bf u}) = {\bf U}$ for some ${\bf U} \gamma {\bf U}_n$, and 

($\star$)  for every  $x, y \in \con ({\bf u})$,  if  $\Theta(x)$  and  $\Theta(y)$ are powers of the same letter then $x=y$.

 Since  ${\mathbf U}_n$ is $xy$-simple, the word  ${\mathbf U}$ is also $xy$-simple.
Consequently, for each $z \in \mul({\bf u})$  either $\Theta(z)=x^p$ or  $\Theta(z)=y_i^q$ for some $p, q \ge 1$ and $1 \le i \le n$.
Property ($\star$) implies the following:

(i)  if  $\Theta(z)=y_i^q$  then $z$ forms one island in $\bf u$;

(ii) if  $\Theta(z)=x^p$ then $z$ forms at most two islands in $\bf u$;

(iii)  if some letter $z$ forms two islands in $\bf u$, then  $\Theta(z)=x^p$ and $z$ is the only letter in $\bf u$
which forms two islands in $\bf u$.

Hence only two cases are possible.

{\bf Case 1}: every letter in $\bf u$ forms only one island in $\bf u$.

In this case, $\bf u$ is a $\gamma$-term for $M$ because $(\mathtt{a^+ \circ_\gamma b^+})$ is a $\gamma$-term for $M$ by Corollary~\ref{C: taum}.

{\bf Case 2}:  $\bf u$ contains a unique letter $z$ which forms two islands  ${_1z^+}$ and ${_2z^+}$    in $\bf u$.

In this case, $\Theta(z)=x^p$. 
Since $\bf U$ is $xy$-simple and $\bf u$ involves less than $n$ letters,
the word $\bf u$ satisfies the following:

(P1)  there is a simple letter $t$ in $\bf u$ between  ${_1z^+}$ and ${_2z^+}$.

Since for each $y \in \mul({\bf u})$ with  $y \ne  z$  we have  $\Theta(y)=y_i^q$, the word  $\bf u$ also satisfies the following:

(P2)  if $y \in \mul({\bf u})$ and $y \ne  z$ then $y$ forms only one island in $\bf u$ and ${\bf u}(z, y) = z^{k_1}y^{k_2}z^{k_3}$ for some $k_1, k_3 \ge 1$ and $k_2 \ge 2$.

Since  both  $(\mathtt{a^+ \circ_\gamma t  \circ_\gamma b^+ \circ_\gamma a^+}) = \{a^{k_1}tb^{k_2}a^{k_3} \mid  k_1, k_3 \ge 1, k_2 \ge 2\}$ and  \\ $(\mathtt{a^+   \circ_\gamma b^+ \circ_\gamma t \circ_\gamma a^+}) = \{a^{k_1}b^{k_2}t a^{k_3} \mid  k_1, k_3 \ge 1, k_2 \ge 2\}$  are $\gamma$-terms for $M$, any word $\bf u$ that satisfies Properties (P1) and (P2)
must be a $\gamma$-term for $M$.

Since $\bf u$  is a $\gamma$-term for $M$   in every case, the monoid $M$ is NFB by Sufficient Condition~\ref{SC}.
\end{proof}

\begin{cor} \label{main2}  The monoid $M_\gamma(\mathtt{a^+b^+ta^+, a^+tb^+a^+})$ is NFB.

\end{cor}

\begin{proof}  Since there is a simple letter $t$ between the two islands formed by $a$ in each word in  $(\mathtt{a^+ \circ_\gamma t  \circ_\gamma b^+ \circ_\gamma a^+})$ and in  $(\mathtt{a^+   \circ_\gamma b^+ \circ_\gamma t \circ_\gamma a^+})$,
it is easy to see that ${\mathbf U}_n=x y_1^2y_2^2\cdots y^2_{n-1} y_n^2x$ is not a $\gamma$-term for
$M=M_\gamma(\mathtt{a^+b^+ta^+, a^+tb^+a^+})$.  For example,  $M$ satisfies 
\[x y_1^2y_2^2\cdots y^2_{n-1} y_n^2x\approx  xy_1^2xy_2^2\cdots y^2_{n-1} y_n^2x\] or
\[x y_1^2y_2^2\cdots y^2_{n-1} y_n^2x\approx x y_1^2y_2^2\cdots y^2_{n-1}y_nxy_n\]
Since  $(\mathtt{a^+ \circ_\gamma t  \circ_\gamma b^+ \circ_\gamma a^+})$ and $(\mathtt{a^+   \circ_\gamma b^+ \circ_\gamma t \circ_\gamma a^+})$, are $\gamma$-terms for $M$ by Proposition~\ref{P: MW}, the monoid $M$ 
is NFB by Sufficient Condition~\ref{SC: for A}.
\end{proof}

Since the monoids $M_\gamma(\mathtt{a^+b^+ta^+, a^+tb^+a^+})$ and $A^1 \times \overline{A}^1$ are equationally equivalent by Theorem~\ref{main1}(iv),
Corollary~\ref{main2} gives us the following.

\begin{cor} \cite[Theorem~6.2]{WTZ2} \label{main3}  The monoid  $A^1 \times \overline{A}^1$  is NFB.

\end{cor}

\section{Sets of 2-island-limited words}

 Let $_{i{\bf u}}x$ denote the $i$th from the left occurrence of $x$ in a word ${\bf u}$. 
For each $k \ge 1$ and  every ${\bf u}, {\bf v} \in \mathfrak A^\ast$   define:  
\begin{itemize}
\item

${\bf u} \lambda_{k} {\bf v}$  $\Leftrightarrow$  for each $i \le k$ the occurrences $_{i\mathbf u}x$ and $_{(i+1)\mathbf u}x$
are adjacent in $\bf u$ if and only if  the occurrences $_{i\mathbf v}x$ and $_{(i+1)\mathbf v}x$ 
are adjacent in $\bf v$.

\end{itemize}
Let $\rho_k$ denote the equivalence relation dual to $\lambda_{k}$.

\begin{fact} For each $k\ge 1$ the relation $\tau_1 \wedge \lambda_k$ is a congruence on $\mathfrak A^\ast$ which partitions $\gamma_k$.
\end{fact}

\begin{proof} First notice that  two  $\tau_1$-related words begin and end with the same letters. Using this fact, it is straightforward to verify
that $\tau_1\wedge \lambda_{k}$ relation is stable under multiplication in  $\mathfrak A^\ast$.

To see that the congruence $\tau_1 \wedge \lambda_k$ partitions $\gamma_k$,
suppose that ${\bf u} (\tau_1\wedge \lambda_{k}) {\bf v}$. Assume that some letter $x$ occurs $1 \le n \le k$ times in $\bf u$.
Since  ${\bf u} \tau_1 {\bf v}$, the number of islands formed by $x$ in $\bf v$ is the same as the number of islands formed by $x$ in $\bf u$.
Since  ${\bf u} \lambda_{k} {\bf v}$ the exponents of the corresponding islands formed by $x$ in $\bf u$ and $\bf v$ are the same. Hence 
$x$ occurs $n$ times in $\bf v$. Since $x$ is arbitrary,  ${\bf u} \gamma_{k} {\bf v}$.
\end{proof}

Notice that $\lambda = \tau_1 \wedge \lambda_1$ and $\rho = \tau_1 \wedge \rho_1$ where $\lambda$ and $\rho$ are the equivalence relations defined in the introduction.

 Given $k\ge 1$ we say that  ${\mathbf u} \in \mathfrak A^*$ is  {\em $k$-island-limited}  if each letter forms at most $k$ islands in $\bf u$. For example, the word $x^7yty^3$ is 2-island-limited. 
If $\tau$ is a congruence that partitions $\tau_1$ and $\mathtt u$ is a $\tau$-class in 
$\mathfrak A^*/\tau$, then it is easy to see that some word in $\mathtt u$ is  $k$-island-limited  if and only if  every word in $\mathtt u$ is  $k$-island-limited. In this case, we say that the $\tau$-word $\mathtt u$ is
 {\em $k$-island-limited}.

 Let $\bf u$ be an 1-island-limited word.
If ${\bf u} \tau {\bf v}$ for some  $\tau \in \{\gamma, \lambda, \rho\}$, then it is easy to see that  ${\bf u} \tau {\bf v}$ for each $\tau \in \{\gamma, \lambda, \rho\}$.
 For example,  $(\mathtt{x^+\circ_\gamma y \circ_\gamma z \circ_\gamma p^+ \circ_\gamma t})  = (\mathtt{x^+\circ_\lambda y \circ_\lambda z \circ_\lambda p^+ \circ_\lambda t})  = (\mathtt{x^+\circ_\rho y \circ_\rho z \circ_\rho p^+ \circ_\rho t})  = 
\{x^nyzp^mt \mid n,m \ge 2\}$ is  a  $\gamma$-class, $\lambda$-class and $\rho$-class.
Thus we have the following.

\begin{fact}\label{F: isom}  Let $\tau \in \{\gamma, \lambda, \rho\}$ and
  $\mathtt{W}$  be a set of  1-island-limited $\tau$-words.
Then  $M_\gamma(\mathtt W) \cong M_\lambda(\mathtt W) \cong M_\rho(\mathtt W)$.

\end{fact}

For example, $M_\gamma(\mathtt{a^+b^+}) \cong  M_\lambda(\mathtt{a^+b^+}) \cong M_\rho(\mathtt{a^+b^+})$.

\begin{lemma}  \label{L: twoislands}

  Let $\mathtt u$ be a 2-island-limited $\lambda$-class of $\mathfrak A^*$  and  ${\mathtt v} \le_{\lambda} {\mathtt u}$. If $\mathtt u$ is a $\lambda$-term  for a monoid variety $\vv$ then ${\mathtt v}$ is also a $\lambda$-term for $\vv$.

\end{lemma}

\begin{proof} If $\mathtt u$ is 1-island-limited then   $\mathtt v$ is also 1-island-limited and both $\mathtt u$ and $\mathtt v$ 
are $\gamma$-classes of $\mathfrak A^*$.  Consequently, $\mathtt v$ is
a $\lambda$-term for $\vv$ by Corollary~\ref{C: taum}.
So, we may assume that some letter forms two islands in some (every) word in $\mathtt u$. Hence every word in $\mathtt u$ involves at least two distinct letters.  

Take some  ${\bf v} \in {\mathtt v}$. Then  ${\mathbf v}$  is a subword of some ${\bf u} \in {\mathtt u}$ by Lemma \ref{subword}.  
Since  $\bf u$ is  a $\gamma$-term for $\vv$  by Observation~\ref{O: twocongr}, the word 
$\bf v$ is  a $\gamma$-term for $\vv$ by Lemma~\ref{twoletters} and Fact~\ref{F: gamma}.

If   ${\bf v}$ is  not a $\lambda$-term for $\vv$, then
${\vv} \models {\bf v} \approx {\bf w}$ such that ${\bf v}\gamma{\bf w}$
but ${\bf v}$ and  ${\bf w}$ are not $\lambda$-related.  Three cases are possible.

{\bf Case 1}:  For some $x \in  \con({\bf v})= \con({\bf w})$ and  $k_1, k_2  \ge 1$,
${\bf v} = {\bf a}x^{k_1}{\bf b}x^{k_2}{\bf c}$,  ${\bf u} = {\bf p}{\bf a}x^{k_1}{\bf b}x^{k_2}{\bf c}{\bf s}$ ,
where  ${\bf p}{\bf a}{\bf b}{\bf c}{\bf s}$ has no occurrences of $x$,  the words ${\bf p}$ and ${\bf c}{\bf s}$ are possibly empty  but $\bf a$ and ${\bf b}$ are not empty, and  one of the following holds:

(i) $k_1 =1$ but the first  island formed by $x$ in ${\bf w}$  contains at least two occurrences of $x$; 

(ii) $k_1 >1$ but the first  island formed by $x$ in ${\bf w}$  contains exactly one occurrences of $x$. 

 Then the words  ${\bf u}$ and  ${\bf p}{\bf w}{\bf s}$ also not $\lambda$-related.
Since $\vv \models {\bf u} =   {\bf p v s} \approx {\bf p}{\bf w}{\bf s}$, this contradicts the fact that $\bf u$ is  a $\lambda$-term for  $\vv$.

{\bf Case 2}: For some letter $x$ and   $q \ge n \ge 1$,  $k, r \ge 1$,    $m \ge 2$,

 ${\bf v} = x{\bf b}x^{n}{\bf c}$,   ${\bf u} = {\bf p}x^k{\bf b}x^q{\bf c}{\bf s}$, ${\bf w} = x^m{\bf b}'x^{r}{\bf c}'$ , where  
 ${\bf p}{\bf b}{\bf c}{\bf s}{\bf b}'{\bf c}'$ has no occurrences of $x$,  the words ${\bf p}$, ${\bf c}{\bf s}$ and ${\bf c}'$ are possible empty but $\bf b$ and ${\bf b}'$ are not empty.

If $\bf v$ and $x {\bf b}'x^{r}{\bf c}'$ are not $\lambda$-related, then  for some  $y \in  \con({\bf v})= \con({\bf w})$ and  $k_1, k_2  \ge 1$,
${\bf v} = {\bf a}_1y^{k_1}{\bf b}_1y^{k_2}{\bf c}_1$ and  ${\bf u} = {\bf p}{\bf a}_1y^{k_1}{\bf b}_1y^{k_2}{\bf c}_1{\bf s}$ so that
if $k_1 =1$ then the first  island formed by $y$ in ${\bf w}$  contains at least two occurrences of $y$, and  if  $k_1 >1$ then  this island
 contains exactly one occurrences of $y$. Hence one can give the same argument as in Case 1 using $y$ instead of $x$.
 So, we may  assume that ${\bf v} \lambda (x {\bf b}'x^{r}{\bf c}')$.

Notice that the identity ${\bf v} \approx {\bf w}$ implies ${\bf u}_1 = {\bf p}x{\bf b}x^n{\bf c}{\bf s} \approx {\bf p} x^m{\bf b}'x^r{\bf c}'{\bf s} = {\bf u}_2$. 

If $k=1$ then ${\bf u} \lambda {\bf u}_1$. Hence ${\bf u}_1 \in \mathtt u$.
Since  $\vv \models {\bf u}_1 \approx {\bf u}_2$ but  ${\bf u}_1$ and  ${\bf u}_2$  are not $\lambda$-related, this contradicts our assumption that every word in $\mathtt  u$ is  a $\lambda$-term for  $\vv$.

If $k \ge 2$ then ${\bf u}  \lambda  {\bf u}_2$  because ${\bf v} \lambda (x {\bf b}'x^{r}{\bf c}')$.  Hence ${\bf u}_2 \in \mathtt u$.
Since  $\vv \models {\bf u}_1 \approx {\bf u}_2$ but ${\bf u}_1$ and  ${\bf u}_2$  are not $\lambda$-related, this contradicts our assumption  that every word in $\mathtt u$ is  a $\lambda$-term for  $\vv$.

{\bf Case 3}:  For some letter $x$ and    $k \ge m \ge 2$,   $n, q, r \ge 1$,

${\bf v} = x^m{\bf b}x^n{\bf c}$,   ${\bf u} = {\bf p}x^k{\bf b}x^q{\bf c}{\bf s}$, ${\bf w} = x{\bf b}'x^r{\bf c}'$, where 
 ${\bf p}{\bf b}{\bf c}{\bf s}{\bf b}'{\bf c}'$ has no occurrences of $x$,   the words ${\bf p}$, ${\bf c}{\bf s}$ and ${\bf c}'$ are possible empty but $\bf b$ and ${\bf b}'$ are not empty.

In this case, the identity ${\bf v} \approx {\bf w}$ implies ${\bf u}_1 =  {\bf p}x^m{\bf b}x^n{\bf c}{\bf s} \approx {\bf p} x{\bf b}'x^r{\bf c}'{\bf s}  = {\bf u}_2$.

Since  ${\bf u} \lambda {\bf u}_1$, the word ${\bf u}_1$ belongs to the $\lambda$-class $\mathtt u$. Since $\vv \models {\bf u}_1 \approx {\bf u}_2$ but
${\bf u}_1$ and  ${\bf u}_2$  are not $\lambda$-related, 
this contradicts our assumption that $\mathtt u$ is  a $\lambda$-term for  $\vv$.

To avoid  a contradiction in every case,  we conclude that $\bf v$ is a $\lambda$-term for $\vv$. Since $\bf v$ is an arbitrary word in $\mathtt v$, the $\lambda$-class $\mathtt v$ is a $\lambda$-term for $\vv$.
\end{proof}

\begin{cor} \label{C: MW2}  Let  $\mathtt{W} \subset \mathfrak A^*/\lambda$ be a set of 2-island-limited $\lambda$-words. Then a monoid variety $\vv$ contains  $M_\lambda(\mathtt {W})$ if and only if every $\lambda$-word in $\mathtt W$ is a $\lambda$-term for $\vv$.
\end{cor}

\begin{proof} This is a consequence of  Proposition~\ref{P: MW} and Lemma~\ref{L: twoislands}. 
It can be justified by repeating the proof of Corollary~\ref{C: MW1} and using  Lemma~\ref{L: twoislands} instead of 
Corollary~\ref{C: taum}.
\end{proof}

\section{Every subvariety of $\mathbb J$ is generated by a monoid of the form $M_{\lambda}({\mathtt W})$} \label{sec:J}

The variety given by a set of identities  $\Sigma$  is denoted by $\var \Sigma$.
 It follows from \cite[Proposition 3.1]{1977} that
\[{\mathbb E} = \var\{ x^2 \approx x^3, xyx \approx x^2y, x^2 y^2 \approx y^2 x^2\}.\]
Following \cite{SG} put:
\[{\mathbb F} = \var\{ x^2 y^2 \approx y^2 x^2, xytxy \approx yxtxy, xtxs \approx xtxsx\};\]
\[{\mathbb H} = \var\{xtx \approx xtx^2,  x^2 y^2 \approx y^2 x^2, xytxy \approx yxtxy, \]
\[xt_1xt_2t_3x \approx xt_1xt_2xt_3x,  x^2yty \approx xyxty \approx yx^2ty \};\]
\[{\mathbb I} = \var\{xtx \approx xtx^2,  x^2 y^2 \approx y^2 x^2, xytxy \approx yxtxy, \]
\[xt_1xt_2t_3x \approx xt_1xt_2xt_3x,  xtxysy \approx xtyxsy \};\]
\[{\mathbb J} = \var\{xtx \approx xtx^2,  x^2 y^2 \approx y^2 x^2, xytxy \approx yxtxy, \]
\[xt_1xt_2t_3x \approx xt_1xt_2xt_3x,  {\bf w}_n(\pi) \approx {\bf w}_n'(\pi) \mid n\ge 0, \pi \in S_n  \},\]
where $ {\bf w}_n(\pi) = x y_{1\pi} y_{2\pi} \dots y_{n\pi} x t_1 y_1 t_2 y_2 \dots t_n y_n$,

${\bf w}'_n(\pi) = x^2 y_{1\pi} y_{2\pi} \dots y_{n\pi}  t_1 y_1 t_2 y_2 \dots t_n y_n$ and
 $S_n$ is the full symmetric group on the $n$-element set.

 Let $F^1$  denote  the monoid  obtained  by adjoining an identity element to the following  semigroup:
\[F = \langle  b, c \mid  b^2 = b^3, c^2=0, cb=c, b^2c=0 \rangle.\]
Notice that the multiplication table of the six-element monoid $F^1 = \{1, b, b^2, c,  bc, 0\}$ is the transpose of the multiplication table which defines the monoid $B$  in \cite[Section 9]{LeeLi}.
According to Proposition 9.1 in
 \cite{LeeLi},  $\{ xtysxy \approx xtysyx, xtxysy \approx xtyxsy, xytxsy \approx yxtxsy, txsx \approx xtxsx\}$ is the identity basis 
for $B$. It is easy to check that $\var B = \overline{\mathbb F}.$
Therefore, $\mathbb F = \var F^1$.

 The following statement is a reformulation of Lemma 3.6 in \cite{SG}.

\begin{fact}\label{ata}

 (i) The $\lambda$-word $\mathtt{atba^+ sb^+}$ is a $\lambda$-term for the variety ${\mathbb J}$.

 (ii) The $\lambda$-word $\mathtt{ba^{+} sb^{+}}$ is a $\lambda$-term for the variety ${\mathbb I}$.

 (iii)  The $\lambda$-word $\mathtt{abta^{+} sb^+}$ is a $\lambda$-term for the variety ${\mathbb H}$.

\end{fact}

 Figure~\ref{pic: J} below duplicates  Figure~1 in \cite{SG} which exhibits the subvariety lattice of  $\mathbb J$.
This lattice contains 
\[ \mathbb B_0  \stackrel{\cite{ELB0}}{=}{\mathbb E} \vee \overline{\mathbb E}
\stackrel{Theorem~\ref{latticeA}(iii)}{=}   \mathbb M_{\gamma}(\mathtt{ta^+, a^+t})  \stackrel{Fact~\ref{F: isom}}{=} \mathbb M_{\lambda}(\mathtt{ta^+, a^+t}).   \]
Comparing  Figure~\ref{pic: A} with Figure~\ref{pic: J} tells us that  $\mathbb B_0$
 is a common subvariety of $\mathbb J$ and $\mathbb A \vee \overline{\mathbb A}$.  
In view of Theorem~\ref{latticeA} and Fact~\ref{F: isom}, every subvariety of $\mathbb B_0$ is generated by a monoid of the form $M_\lambda(W)$.
 The following theorem describes the generating monoid for each of the five subvarieties of $\mathbb J$ which is not a subvariety of  $\mathbb B_0$.

\begin{theorem}\label{main}

(i)  ${\mathbb F} = \mathbb M_{\lambda}(\mathtt{ata^+})$.

(ii) ${\mathbb F} \vee \overline{\mathbb E} = \mathbb M_{\lambda}(\mathtt{ata^+, a^+t}) = \mathbb M_{\lambda}(\mathtt{a^+ta^+})$.

(iii) $\mathbb H = \mathbb M_{\lambda}(\mathtt{abta^{+}sb^{+}})$.

(iv) The variety $\mathbb I$ is generated by the 19-element monoid $M_{\lambda}(\mathtt{ba^{+}sb^{+}})$.

(v) The variety $\mathbb J$ is generated by the 31-element monoid $M_{\lambda}(\mathtt{atba^{+}sb^{+}})$.

\end{theorem}

\begin{proof} (i)  The variety $\mathbb F$ is a cover of $\mathbb E$ by Figure~\ref{pic: J}.
Since \[{\mathbb E} \stackrel{Theorem~\ref{latticeA}(i)}{=} \mathbb M_{\gamma}(\mathtt{ta^+}) \stackrel{ Fact~\ref{F: isom}}{=} \mathbb M_{\lambda}(\mathtt{ta^+}),\]
 the $\lambda$-word  $(\mathtt{t \circ_\lambda a^+})$ is  a $\lambda$-term for ${\mathbb F}$ by Corollary~\ref{C: MW2}.

Suppose that for some $n\ge1$, the monoid $F^1$ satisfies an identity $xtx^n \approx {\bf u}$ such that $xtx^n$ and $\bf u$ are not $\lambda$-related.  Since  $(\mathtt{t \circ_\lambda a^+})$ is a $\lambda$-term for $\mathbb F$, there are only two possibilities for $\bf u$: 
either ${\bf u} = x^mt$ for some $m>1$ or ${\bf u} = x^m tx^n$ for some  $m >1$, $n\ge1$.  
But neither  $xtx^n \approx x^mt$ nor  $xtx^n \approx x^mtx^n$ holds in $F^1$, 
because substituting $b$ for $x$ and $c$ for $t$ turns both identities into $bc = 0$.
Therefore,  $(\mathtt{a \circ_\lambda  t \circ_\lambda a^+})$ is a $\lambda$-term for $\mathbb F$.
Consequently,  ${\mathbb F}$ contains $M_{\lambda}(\mathtt{ata^+})$ by Corollary~\ref{C: MW2}.

Since ${\mathbb E} \models xyx \approx x^2y$,
the $\lambda$-word $\mathtt{ata^+}$ is not a $\lambda$-term for $\mathbb E$. Corollary~\ref{C: MW2}
 implies that $\mathbb E$ does not contain $M_{\lambda}(\mathtt{ata^+})$.
Since $\mathbb E$ is a unique maximal subvariety of $\mathbb F$, the variety ${\mathbb F}$ is generated by the monoid $M_{\lambda}(\mathtt{ata^+})$.

(ii) Since ${\overline{\mathbb E}} \stackrel{Theorem~\ref{latticeA}(ii)}{=} \mathbb M_{\gamma}(\mathtt{a^+t}) \stackrel{ Fact~\ref{F: isom}}{=} \mathbb M_{\lambda}(\mathtt{a^+t})$, in view of Part (i) and  Corollary \ref{product} we have
 ${\mathbb F} \vee \overline{\mathbb E} = \mathbb M_{\lambda}(\mathtt{ata^+, a^+t})$.

Let us verify that  $\mathbb M_{\lambda}(\mathtt{ata^+, a^+t}) = \mathbb M_{\lambda}(\mathtt{a^+ta^+}).$

Indeed, since $(\mathtt{a  \circ_\lambda t  \circ_\lambda a^+}) \le_\lambda (\mathtt{a^+  \circ_\lambda t  \circ_\lambda a^+})$ and 
$(\mathtt{a^+  \circ_\lambda t}) \le_\lambda (\mathtt{a^+  \circ_\lambda t  \circ_\lambda a^+})$,  both 
 $(\mathtt{a  \circ_\lambda t  \circ_\lambda a^+})$ and  $(\mathtt{a^+  \circ_\lambda t})$ are $\lambda$-terms for 
$ M_{\lambda}(\mathtt{a^+ta^+})$  by Lemma~\ref{L: twoislands}. Hence \\
$\mathbb M_{\lambda}(\mathtt{a t a^+, a^+ t}) \subseteq \mathbb M_{\lambda}(\mathtt{a^+ta^+})$ by Corollary~\ref{C: MW2}.

 Conversely, since $(\mathtt{t}) \le_\lambda (\mathtt{t  \circ_\lambda a^+}) \le_{\lambda} (\mathtt{a  \circ_\lambda t  \circ_\lambda a^+})$, 
 the word $t$ is an isoterm and
$(\mathtt{t  \circ_\lambda a^+})$ is  a $\lambda$-term for  $M_{\lambda}(\mathtt{a t a^+, a^+ t})$
 by Lemma~\ref{L: twoislands}.

If $\vv \models x^ntx^m \approx {\bf u}$ for some $n \ge 2, m\ge 1$ then ${\bf u}= x^ktx^p$ for some $k+p\ge 2$ by Observation~\ref{O: gammak}.
Since $\mathtt{ata^+, ta^+}$ and $\mathtt{a^+t}$ are $\lambda$-terms for $\vv$, the only possibility is that $k\ge 2$ and $p\ge 1$.
Hence $(\mathtt{a^+\circ_\lambda t \circ_\lambda a^+})$ is  a $\lambda$-term for $\vv$.

Therefore, $\mathbb M_{\lambda}(\mathtt{a t a^+, a^+ t}) \supseteq \mathbb M_{\lambda}(\mathtt{a^+ta^+})$ by Corollary~\ref{C: MW2}.

(iii)  The $\lambda$-word $\mathtt{abta^{+}sb^{+}}$ is a $\lambda$-term for $\mathbb H$ by Fact \ref{ata}. 
Therefore, the variety ${\mathbb H}$ contains $M_{\lambda}(\mathtt{abta^{+}sb^{+}})$ by Corollary~\ref{C: MW2}.

The variety ${\mathbb F} \vee \overline{\mathbb E}$ satisfies $xytxsy  \approx yxtxsy$ by  Lemma 2.8 in \cite{SG}.
Thus the $\lambda$-word $\mathtt{abta^{+}sb^{+}}$ is not a $\lambda$-term for ${\mathbb F} \vee \overline{\mathbb E}$ and Corollary~\ref{C: MW2}  implies that the variety ${\mathbb F} \vee \overline{\mathbb E}$ does not contain $M_{\lambda}(\mathtt{abta^{+}sb^{+}})$.
Since ${\mathbb F} \vee \overline{\mathbb E}$ is a unique maximal subvariety of $\mathbb H$, the variety ${\mathbb H}$ is generated by  $M_{\lambda}(\mathtt{abta^{+}sb^{+}})$.

(iv) The $\lambda$-word $\mathtt{ba^{+}sb^{+}}$ is a $\lambda$-term for $\mathbb I$ by Fact \ref{ata}.
 Therefore, the variety ${\mathbb I}$ contains $M_{\lambda}(\mathtt{ ba^{+}sb^{+} })$ by Corollary~\ref{C: MW2}.

By the definition, the variety $\mathbb H$ satisfies $yx^2sy \approx  x^2ysy$.
Thus the $\lambda$-word $\mathtt{ba^{+}sb^{+}}$ is not a $\lambda$-term for ${\mathbb H}$, and Corollary~\ref{C: MW2}  implies that the variety ${\mathbb H}$ does not contain $M_{\lambda}(\mathtt{ ba^{+}sb^{+}})$.
Since ${\mathbb H}$ is a unique maximal subvariety of $\mathbb I$, the variety $\mathbb I$ is generated by $M_{\lambda}(\mathtt{ba^{+}sb^{+}})$.

(v) The $\lambda$-word  $\mathtt{atba^{+}sb^{+}}$ is a $\lambda$-term for $\mathbb J$ by
Fact \ref{ata}.
Therefore, the variety ${\mathbb J}$ contains $M_{\lambda}(\mathtt{atba^{+}sb^{+}})$ by Corollary~\ref{C: MW2}.

 By the definition, the variety $\mathbb I$
satisfies $xtxysy \approx xtyxsy$.
Thus, $\mathtt{atba^{+}sb^{+}}$ is not a $\lambda$-term for $\mathbb I$. Corollary~\ref{C: MW2}  implies that
the variety $\mathbb I$ does not contain $M_{\lambda}(\mathtt{atba^{+}sb^{+}})$. 
Since $\mathbb I$ is a unique maximal subvariety $\mathbb J$, the variety $\mathbb J$ is generated by $M_{\lambda}(\mathtt{atba^{+}sb^{+}})$.
\end{proof}

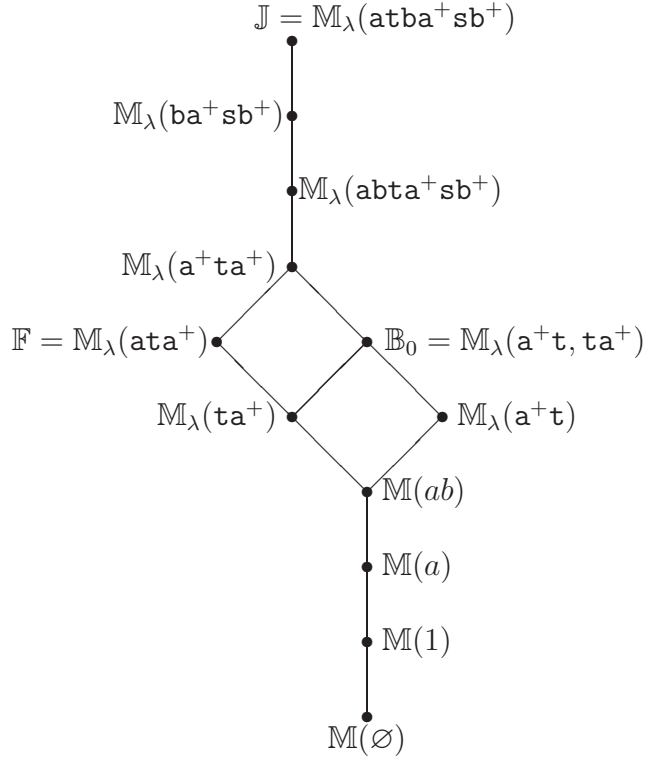
\begin{figure}[htb]
\unitlength=1mm
\linethickness{0.4pt}
\begin{center}
\begin{picture}(55,100)
\put(35,5){\circle*{1.33}}
\put(35,15){\circle*{1.33}}
\put(35,25){\circle*{1.33}}
\put(35,35){\circle*{1.33}}
\put(45,45){\circle*{1.33}}
\put(25,45){\circle*{1.33}}
\put(35,55){\circle*{1.33}}
\put(15,55){\circle*{1.33}}
\put(25,65){\circle*{1.33}}
\put(25,75){\circle*{1.33}}
\put(25,85){\circle*{1.33}}
\put(25,95){\circle*{1.33}}

\put(35,5){\line(0,1){30}}
\put(35,35){\line(-1,1){20}}
\put(35,35){\line(1,1){10}}
\put(25,45){\line(1,1){10}}
\put(25,45){\line(1,1){10}}
\put(15,55){\line(1,1){10}}
\put(45,45){\line(-1,1){20}}
\put(25,65){\line(0,1){30}}


\put(35,2){\makebox(0,0)[cc]{${\mathbb M}(\varnothing)$}}
\put(37,15){\makebox(0,0)[lc]{${\mathbb M}(1)$}}
\put(37,25){\makebox(0,0)[lc]{${\mathbb M}(a)$}}
\put(37,35){\makebox(0,0)[lc]{${\mathbb M}(ab)$}}
\put(23,45){\makebox(0,0)[rc]{${\mathbb M}_\lambda(\mathtt{ta^+})$}}
\put(14,55){\makebox(0,0)[rc]{$\mathbb F = {\mathbb M}_\lambda(\mathtt{ata^+})$}}
\put(72,55){\makebox(0,0)[rc]{$\mathbb B_0 = {\mathbb M}_\lambda(\mathtt{a^+t, ta^+})$}}
\put(23,65){\makebox(0,0)[rc]{${\mathbb M}_\lambda(\mathtt{a^+ta^+})$}}
\put(53,75){\makebox(0,0)[rc]{${\mathbb M}_\lambda(\mathtt{abta^+sb^+})$}}
\put(24,85){\makebox(0,0)[rc]{${\mathbb M}_\lambda(\mathtt{ba^+s b^+})$}}
\put(47,45){\makebox(0,0)[lc]{${\mathbb M_\lambda(\mathtt{a^+t})}$}}
\put(20,98){\makebox(0,0)[lc]{${\mathbb J} = {\mathbb M}_\lambda(\mathtt{atba^+sb^+})$}}
\end{picture}
\end{center}
\caption{The subvariety lattice of $\mathbb J$ (cf. Figure~1 in \cite{SG})}
\label{pic: J}
\end{figure}

\begin{remark} \label{B0}  The monoids in Theorem \ref{main} are not the minimal generating monoids for their varieties.

(i) The monoid $F^1$ is isomorphic to the submonoid $\{1, \mathtt{a, a^+, ta^+, ata^+}, 0\}$ of $M_\lambda(\mathtt{ata^+})$
by $b \rightarrow \mathtt{a}, c \rightarrow \mathtt{ta^+}$.

(ii)   Gusev noticed that  $M_{\lambda}(\mathtt{abta^{+}sb^{+}})$ contains a proper  submoid $M_1$ such that 
${\mathbb H} = \var M_1$. The monoid $M_1$ is generated by $\{\mathtt{a,b,ta^+, a^+sb^+}\}$ and has 14 elements.

(iii) The computer program Prover9 gives the multiplication table of an 8-element monoid which generates the variety $\mathbb I$
(Lee, July 2019). This multiplication table corresponds to
the submonoid of $M_{\lambda}(\mathtt{ba^{+}sb^{+}})$ which consists of
\[\{1, \mathtt{ba^{+}, a^+, a^+sb^+, ba^+sb^+, b, b^+},  0 \}.\]

(iv)  Gusev noticed that  $M_{\lambda}(\mathtt{atba^{+}sb^{+}})$ contains a proper  submoid $M_2$ such that 
${\mathbb J} = \var M_2$. The monoid $M_2$ is generated by $\{\mathtt{a,b,t, a^+sb^+}\}$ and has 19 elements.

\end{remark}

\section{Representations of the monoids $\mathfrak A^*/ \tau_1$, $\mathfrak A^*/ \gamma$, $\mathfrak A^*/ \lambda$ and $\mathfrak A^*/ \rho$}\label{sec:tau}

Let  $\tau$ be a congruence on $\mathfrak A^*$ and $\mathtt W \subseteq \mathfrak A^*/\tau$ be a set of $\tau$-words.
In view of \eqref{mult in MW}, the elements of  $M_\tau(\mathtt W)$ are computed in terms of the quasi-order $\le_\tau$ on
$\mathfrak A^*/\tau$ and
the multiplication in $M_\tau(\mathtt W)$ is  in terms of the operation $\circ_\tau$ in  $\mathfrak A^*/\tau$.

In this section,  for each  $\tau \in \{\tau_1, \gamma, \lambda, \rho\}$, we provide an algorithm for calculating the relation $\le_\tau$
and the  operation $\circ_\tau$ in  $\mathfrak A^*/\tau$. In particular, we show that in this case, the relation  $\le_\tau$ is an order on
$\mathfrak A^*/\tau$. Therefore,  $\mathfrak A^*/\tau$ is a $J$-trivial monoid for  each  $\tau \in \{\tau_1, \gamma, \lambda, \rho\}$.

As we mentioned in the introduction, each element of $\mathfrak A^*/\tau$ can be represented in multiple ways as a word in the alphabet 
$\{\mathtt{a^{1+}} \mid a \in \mathfrak A\}$ if $\tau=\tau_1$ or as a word in the alphabet $\{\mathtt{a}, \mathtt{a^{2+}}  \mid a \in \mathfrak A\}$ if $\tau \in \{\gamma, \lambda, \rho\}$. In this section, we show that each $\tau$-class can be represented by a unique (`canonical")  word
in  this alphabet.

To this aim,  we fix a formal alphabet $\mathfrak B$  which consists of symbols $a^+$ for each $a \in \mathfrak A$ and
for each $a \in \mathfrak A$ consider the following rewriting rules on $(\mathfrak A \cup \mathfrak B)^+$:
\begin{itemize}
\item $R^{\tau_1}_{a\rightarrow a^+}$ replaces an occurrence of $a$ in $\bf u$ by $a^+$;
\item  $R^\gamma_{a\rightarrow a^+}$ replaces an occurrence of $a$ in $\bf u$ by $a^+$ only in case if $\bf u$ contains either another occurrence of $a$ or an occurrence of $a^+$;

\item  $R^\lambda_{a\rightarrow a^+}$ replaces an occurrence of $a$ in $\mathbf u$ by $a^+$ only in case if either $a$ or $a^+$ appears in $\bf u$ to the left of this occurrence of $a$;

\item  $R^\rho_{a\rightarrow a^+}$ replaces an occurrence of $a$ in $\bf u$ by $a^+$ only in case if either $a$ or $a^+$ appears in $\bf u$ to the right of this occurrence of $a$;

\item  $R_{a^+a^+\rightarrow a^+}$ replaces $a^+a^+$ by $a^+$;

\item  $R_{aa^+\rightarrow a^+}$ replaces $aa^+$ by $a^+$.

\item  $R_{a^+a\rightarrow a^+}$ replaces $a^+a$ by $a^+$.
\end{itemize}

\begin{lemma} \label{reduced} 

If  $\tau \in \{\tau_1, \gamma, \lambda, \rho\}$ then  using the rules 
\begin{equation} \label{rs} R_\tau = \{R^\tau_{a\rightarrow a^+},  R_{a^+a^+\rightarrow a^+}
R_{aa^+\rightarrow a^+}, R_{aa^+\rightarrow a^+}\mid a \in \mathfrak A\}\end{equation}
 in any order,  every word
 ${\bf u} \in (\mathfrak A \cup \mathfrak B)^*$
can be transformed to a unique word $r_\tau({\bf u})$ such that none of these rules is applicable to  $r_\tau({\bf u})$.

\end{lemma}

\begin{proof} Notice that the rules $R^{\tau_1}_{a\rightarrow a^+}$, $R_{a^+a^+\rightarrow a^+}$,  $R_{aa^+\rightarrow a^+}$  and $R_{a^+a \rightarrow a^+}$ are traditional relations on a free monoid while the rules $R^\gamma_{a\rightarrow a^+}$,
$R^\lambda_{a\rightarrow a^+}$ and $R^\rho_{a\rightarrow a^+}$
require to check  certain property of a word $\bf u$ before being applied to $\bf u$. But since an application of any of these rules
does not change this property,  the  rewriting system $\langle (\mathfrak A \cup \mathfrak B)^* \mid R_\tau \rangle$
 is confluent for each  $\tau \in \{\tau_1, \gamma, \lambda, \rho\}$.
\end{proof}

For each  $\tau \in \{\tau_1, \gamma, \lambda, \rho\}$ let $\mathcal R_{\tau}$ denote the set of all words in $(\mathfrak A \cup \mathfrak B)^*$ to which none of the rewriting rules \eqref{rs} is applicable, that is,
${\mathbf u} \in \mathcal R_\tau$ if and only if ${\mathbf u}=r_\tau({\mathbf u})$.
For each  pair of words ${\mathbf u}, {\mathbf v} \in    \mathcal R_\tau$ define:
$${\mathbf u}  {\diamond{_\tau}} {\mathbf v} = r_\tau (\mathbf {uv}),$$ 
where $\mathbf {uv}$ is the result of concatenation
of  ${\mathbf u}$ and  ${\mathbf v}$ in  $(\mathfrak A \cup \mathfrak B)^*$. For example:

   $(a^+ c a^+), (ab^+ t c) \in \mathcal R_ \gamma$  and  $(a^+ c a^+)  { \diamond{_\gamma}} (ab^+ t c) = a^+c^+a^+b^+ t c^+  \in \mathcal R_ \gamma$;

  $(a c a^+), (ab^+ t c) \in \mathcal R_ \lambda$  and  $(a c a^+)  { \diamond{_\lambda}} (ab^+ t c) = aca^+b^+ t c^+  \in \mathcal R_ \lambda$.

Consider the map

  $r_\tau:  (\mathfrak A \cup \mathfrak B)^* \rightarrow  \mathcal R_\tau$ given by ${\mathbf u} \rightarrow r_\tau({\mathbf u})$.

Lemma \ref{reduced} implies that 
the map $r_\tau$ is a homomorphism from  $(\mathfrak A \cup \mathfrak B)^*$ to the monoid $\langle \mathcal R_\tau,  {\diamond{_\tau}} \rangle$. We are interested only in the restriction of  $r_\tau$ to the free monoid $\mathfrak A^*$.

 \begin{prop}  \label{congruence}  Let $\tau_1, \gamma, \lambda$ and $\rho$ be the equivalence relations on  $\mathfrak A^*$ defined in the introduction. Then  for each  $\tau \in \{\tau_1, \gamma, \lambda, \rho\}$:

(i) $\tau$ is the kernel of the  homomorphism

$r_\tau$: $\mathfrak A ^*  \rightarrow  \mathcal R_\tau$ given by ${\mathbf u} \rightarrow r_\tau({\mathbf u})$;

(ii) the homomorphism $r_\tau$: $\mathfrak A ^*  \rightarrow  \mathcal R_\tau$ induces the isomorphism $r_\tau$ 
from $\langle  \mathfrak A^*/ \tau,  {\circ{_\tau}} \rangle$ onto  $\langle \mathcal R_\tau,  {\diamond{_\tau}} \rangle$
defined by  $r_\tau({\mathtt u}) = r_\tau({\bf u})$ for each    ${\mathtt u} \in  \mathfrak A^*/ \tau$ and   ${\bf u} \in {\mathtt u}$.

\end{prop}

\begin{proof} (i)  If  ${\bf u}, {\bf v} \in {\mathfrak A}^*$ 
then Lemma~\ref{reduced} implies that ${\bf u} \tau {\bf v}$ if and only if $r_{\tau}({\bf u}) = r_{\tau}({\bf v})$.

Part (ii) immediately follows from Part (i).
\end{proof}

 Since  $\mathcal R_{\tau_1} \subset  {\mathfrak B}^*$,  we have ${\mathtt v} \le_{\tau_1} {\mathtt u}$ if and only if   $r_{\tau_1}({\mathtt v})$ is a subword of $r_{\tau_1}({\mathtt u})$. Hence for each ${\mathtt u} \in \mathfrak A/{\tau_1}$ we have 
\[r_{\tau_1} (\{{\mathtt u}\}^{\le_{\tau_1}}) = \{ {\mathbf v} \in   {\mathfrak B}^* \mid {\mathbf v} \le r_{\tau_1}({\mathtt u})\}.\]
 In particular, $\{{\mathtt u}\}^{\le_{\tau_1}}$ is a finite set. For example, 
\[r_{\tau_1} (\{\mathtt{a^{1+} \circ_{\tau_1} b^{1+}}\}^{\le_{\tau_1}}) = r_{\tau_1} (\{a^kb^n \mid k, n \ge 0 \}) =  \{ 1, a^+, b^+, a^+b^+ \}.\]
Using Lemma \ref{subword}, one can compute that 
\[r_{\gamma} (\{\mathtt{a \circ_\gamma b^{2+}}\}^{\le_{\gamma}}) = \{ 1, a, b, b^+, ab, ab^+ \}.\]
Notice that the set $r_{\gamma} (\{\mathtt{a \circ_\gamma b^{2+}}\}^{\le_{\gamma}})$ is larger than the set of all subwords of  $r_{\gamma} (\mathtt{a \circ_\gamma b^+}) = r_{\gamma}(\{ab^n \mid k\ge 2\}) = ab^+$.
However, the set $r_{\gamma} (\{\mathtt{a \circ_\gamma b^{2+}}\}^{\le_{\gamma}})$ is finite because for each $k\ge 2$, we have
$r_\gamma(b^2) = r_\gamma(b^k)$  and $r_\gamma(ab^2) = r_\gamma(ab^k)$.

In order to show that for each $\tau \in \{\gamma, \lambda, \rho\}$ and every ${\mathtt u} \in \mathfrak A^*/\tau$ the set
$\{{\mathtt u}\}^{\le_{\tau}}$ is always finite, we  use $K_\tau({\mathtt u})$ to denote the (finite) subset of 
$\{{\bf u} \in  \mathfrak A^* \mid   {\bf u} \in {\mathtt u} \}$
 membership to which is determined as follows:

$\bullet$ Given ${\mathbf  u} \in {\mathtt u}$, the word $\mathbf u$ belongs to $K({\mathtt u})$ if and only if ${\mathbf  u}$ does not have any subwords of the form $x^3$ for any $x \in \mathfrak A$.

 For example, $K_\gamma(\mathtt{a^{2+} \circ_\gamma b  \circ_\gamma a^{2+}}) = \{ aba, a^2ba^2, a^2ba, a^2ba \}$, \\ $K_\lambda(\mathtt{a \circ_\gamma b^{2+}  \circ_\gamma a^{2+}}) = \{ ab^2a, ab^2a^2 \}$.

The following lemma gives us a simple algorithm for computing the set $\{{\mathtt u}\}^{\le_{\tau}}$.

\begin{lemma}  \label{tausubwords}  Let $\tau \in \{\gamma, \lambda, \rho\}$ and  ${\mathtt u} \in \mathfrak A^*/\tau$. Then
\[ r_\tau(\{{\mathtt u}\}^{\le_{\tau}}) = \{ r_{\tau}({\mathbf v}) \mid  {\mathbf v}\le {\mathbf u}, {\mathbf u} \in K({\mathtt u})\}.\]

In particular, the set $\{{\mathtt u}\}^{\le_\tau}$ is finite.
\end{lemma}

\begin{proof} In view of  Lemma \ref{subword}, we have
 \[ r_\tau(\{{\mathtt u}\}^{\le_{\tau}}) = \{ r_{\tau}({\mathbf v}) \mid  {\mathbf v}\le {\mathbf u}, {\mathbf u} \in {\mathtt u}\}.\]
Take ${\mathbf u} \in  {\mathtt u}$. Let ${\mathbf u}'$ denote the word obtained from ${\mathbf u}$ by applying the relation $x^3 = x^2$ for each $x \in \con({\mathbf u})$ from left to right whenever it is applicable. For example, if ${\mathbf u} = x^4yxy^5$ then ${\mathbf u}' = x^2yxy^2$.  Since $r_\tau({\mathbf u}) = r_\tau({\mathbf u}')$,  in order to find all the elements in $\{{\mathtt u}\}^{\le_{\tau}}$ it is sufficient to look only at  subwords of 
words in  $K({\mathtt u})$. Since $K({\mathtt u})$ is finite, the set $\{{\mathtt u}\}^{\le_\tau}$ is also finite and can be effectively computed.
\end{proof}

The next example lists all non-zero elements of the 31-element monoid \\ $M_\lambda(\{\mathtt{atba^{+}sb^{+}}\})$ which generates the limit variety $\mathbb J$ (see Theorem \ref{main}).

\begin{ex} \label{E: J}
\[  r_\tau(\{ \mathtt{atba^{+}sb^{+}}\}^{\le_\lambda}) = \{1,  a, a^+, b, b^+, t, s,\]
\[ at, tb, ba, ba^+, as, a^+s, sb, sb^+,\]
\[ atb, tba, tba^+,  asb, a^+sb, a^+sb^+,\]
\[atba^+, tbas, tba^+s, basb^+, ba^+sb^+,\]
\[atba^+s, tbasb^+, tba^+sb^+, atba^{+}sb^{+} \}\]
\end{ex}

\begin{prop} \label{P: Jtriv} If $\tau \in \{\tau_1, \gamma, \lambda, \rho\}$ then:

(i)  ${\mathtt u}\in \mathfrak A^* / \tau$ is an idempotent of $\mathfrak A^* / \tau$ if and only if  $r_{\tau}({\mathtt u}) =1$  or  $r_{\tau}({\mathtt u}) = x^+$ for some $x \in \mathfrak A$;

(ii)  the relation $\le_\tau$ is an order on $\mathfrak A^* / \tau$, thus $\mathfrak A^* / \tau$  is a $J$-trivial monoid;

(iii) if  ${\mathtt W} \subset \mathfrak A^* / \tau$ is  finite set of $\tau$-words then  $M_{\tau}({\mathtt W})$ is  finite $J$-trivial monoid.

\end{prop}

\begin{proof} (i) If $x \in \mathfrak A$ then $x^+  {\diamond{_\tau}} x^+ = r_\tau (x^+ x^+) = x^+$, $x  {\diamond{_\tau}} x = r_\tau (x^2) = x^+$.

If  $r_\tau ({\mathtt u}) \in \mathcal R_\tau$ contains more than one letter then the word $r_\tau ({\mathtt u})  {\diamond{_\tau}} r_\tau ({\mathtt u}) = r_\tau (\mathtt {u^2})$ is longer than  $r_\tau ({\mathtt u})$.

(ii)  If  ${\mathtt v} \le_{\tau} {\mathtt u}$  then  in view of  Lemma \ref{tausubwords},
either  $r_\tau({\mathtt v}) \in \mathcal R_\tau$ is a subword of $r_\tau({\mathtt u})  \in \mathcal R_\tau$ or $r_\tau({\mathtt v})$  becomes a subword of $r_\tau({\mathtt u})$ after some letters $\{a_1, \dots, a_k \mid k>0\} \subseteq \con(r_\tau({\mathtt v}))$ are replaced by  $a^+_1, \dots, a^+_k$.

This implies that if  ${\mathtt v} \le_{\tau} {\mathtt u}$ and  ${\mathtt u} \le_{\tau} {\mathtt v}$ then
${\mathtt v} = {\mathtt u}$.

(iii) Since $M_{\tau}({\mathtt W})$ is  a Rees quotient of  $J$-trivial monoid   $\mathfrak A^* / \tau$, it is also $J$-trivial.
If $\mathtt W$ is finite then ${\mathtt W} ^{\le_\tau}$ is also finite by Lemma \ref{tausubwords}.
\end{proof}

In contrast, here are some examples of non-J-trivial  monoids of the form $M_\tau(W)$.

\begin{ex}\label{E: C group}  Let $\vv$ be a commutative monoid  variety  containing the 2-element semilattice.  
Then $\vv = \mathbb M_\tau (\{a\}^*)$, where 
$\tau$ is either the trivial congruence or $\tau = \gamma_k \wedge \tau_m$ for some $k, m \ge 0$. 
\end{ex}

\begin{proof} It is well-known and easily verified that every variety of commutative monoids coincides with  $\mathbb C_{n,m} = \var \{x^n \approx x^{n+m}, xy \approx yx\}$ for some $n, m \ge 0$.  It is also well-known that  the variety of all commutative monoids $\mathbb C_{n,0}$ is generated by an infinite cyclic monoid $M (\{a\}^*)$. So, we may assume that $m \ge 1$. Since $\vv$ contains  the 2-element semilattice,
$n=k+1$ for some $k \ge 0$.

It is well-known that $\mathbb C_{k+1, 1}$ is generated by  $M(a^{k})= \{1, a, \dots, a^k, 0\}$.
 In view of Observation~\ref{O: gammak},  every word is a $\gamma_k$-term for $\mathbb M(a^k)$. 
Hence  $\mathbb M_{\gamma_k} (\mathfrak A^*)  \subseteq \mathbb M(a^k)  $ by Proposition~\ref{P: MW}. On the other hand,
$a^k$ is an isoterm for $\mathbb M_{\gamma_k} (\{a\}^*)$. (Otherwise,   $\mathbb M_{\gamma_k}(\{a\}^*) \models x^k \approx x^{k+p}$ for some $p>0$ which contradicts the fact that $x^k$ must be a $\gamma_k$-term for $\mathbb M_{\gamma_k} (\{a\}^*)$.  Hence $\mathbb M(a^k) \subseteq \mathbb M_{\gamma_k}  (\{a\}^*)$ by Proposition~\ref{P: MW}.
Since $\mathbb M_{\gamma_k} (\mathfrak A^*) \supseteq \mathbb M_{\gamma_k}  (\{a\}^*)$, we have $\mathbb M_{\gamma_k} (\mathfrak A^*) =\mathbb M(a^k) =\mathbb M_{\gamma_k}  (\{a\}^*)$.  Therefore
\begin{equation}\label{k1}  \mathbb C_{k+1, 1} = \mathbb M_{\gamma_k} (\{a\}^*) .\end{equation}
It is well-known that $\mathbb C_{1,m} = \mathbb {SL} \vee \mathbb A_m$, where $\mathbb {SL}$ is the variety of all commutative idempotent monoids and $\mathbb A_m$   is the variety of Abelian groups of exponent dividing $m$.
Notice that  $M_{\tau_m}(\{a\}^*)$  is isomorphic to the cyclic group of order $m$  with zero and unity elements adjoined. Therefore,   $\mathbb M_{\tau_m}(\{a\}^*)$ contains both  $\mathbb{SL}$ and $\mathbb A_m$.
On the other hand,  every word in $\{a\}^*$   is a   $\tau_m$-term for $\mathbb C_{1,m} = \mathbb{SL} \vee \mathbb A_m$.
Consequently,  
\begin{equation} \label{m1}
\mathbb C_{1,m} =  \mathbb M_{\tau_m}(\{a\}^*).
\end{equation}
by Proposition~\ref{P: MW}.
Finally,
\[\mathbb C_{k+1,m} = \mathbb C_{k+1,1}  \vee \mathbb A_m =\mathbb C_{k+1,1}  \vee  \mathbb{SL} \vee  \mathbb A_m =
\mathbb C_{k+1,1}  \vee  \mathbb C_{1,m} =\]
\[\stackrel{\eqref{k1}, \eqref{m1}}{=}  \mathbb M_{\gamma_k} (\{a\}^*) \vee  \mathbb M_{\tau_m}(\{a\}^*) \stackrel{\eqref{e: product2}}{=} \mathbb M_{\gamma_k \wedge \tau_m} (\{a\}^*).\] \end{proof}

\subsection*{Acknowledgement} The author thanks Sergey Gusev for reading the previous versions with many thoughtful comments and for finding $M_\gamma(W)$ representation for the monoid $A^1$.
The author is also grateful to Edmond Lee for helpful discussions.

\end{document}